\documentclass[a4paper, reqno, twoside]{amsart}

\usepackage{lmodern}
\usepackage[english]{babel} 
\usepackage{microtype}
\usepackage{amsthm,amsfonts,amssymb,amsmath,mathrsfs}
\usepackage{amsxtra}
\usepackage[all]{xy}
\usepackage{enumitem}
\usepackage{a4wide}
\usepackage[bookmarks=false]{hyperref} 
\usepackage{cleveref}
\usepackage[usenames,dvipsnames,svgnames,table]{xcolor}
\usepackage{version}
\usepackage{txfonts}

\theoremstyle{plain}
\newtheorem{theorem}{\sc Theorem}[section]
\newtheorem{proposition}[theorem]{\sc Proposition}

\newtheorem{lemma}[theorem]{\sc Lemma}
\newtheorem{corollary}[theorem]{\sc Corollary}
\newtheorem*{theorem*}{Theorem}

\theoremstyle{definition}
\newtheorem{definition}[theorem]{\sc Definition}

\newtheorem{example}[theorem]{\sc Example}

\theoremstyle{remark}

\newtheorem{remark}[theorem]{\sc Remark}

\newenvironment{invisible}[1][\unskip]{
	\noindent
	\color{red}
	[{\textbf{\color{blue}TBH}: \textit{#1}}
}{]}

\IfFileExists{tcilatex.tex}{\input{tcilatex}}{}

\newcommand{\tensor}[1]{\otimes_{\scriptscriptstyle{#1}}}

\newcommand{\fk}[1]{\mathfrak{#1}}

\renewcommand{\hom}[3]{\mathrm{Hom}_{\Sscript{#1}}\left(#2,\,#3\right)}

\newcommand{\bd}[1]{\boldsymbol{#1}}

\newcommand{\bara}[1]{\overline{#1}}

\newcommand{\End}[2]{\mathrm{End}_{\Sscript{#1}}(#2)}

 \newcommand{\id}{\mathrm{Id}}

\newcommand{\injlimit}[2]{\varinjlim_{#1}\left({#2}\right)}

\newcommand{\ann}[1]{\mathsf{Ann}\left({#1}\right)}
\renewcommand{\ker}[1]{\mathrm{ker}\left({#1}\right)}
\newcommand{\what}[1]{\widehat{#1}}

\newcommand{\Lin}[1]{\cL{in}_{#1}}
\newcommand{\dlin}[1]{\cD{lin}_{#1}}
\definecolor{bostonuniversityred}{rgb}{0.8, 0.0, 0.0}

\newcommand{\ie}{i.e.~}

\newcommand{\Hom}[6]{{_{#1}^{#2}\mathsf{Hom}_{#3}^{#4}}\left({#5},{#6}\right)}
\newcommand{\algk}{\mathsf{Alg}}
\newcommand{\dalgk}{\mathsf{DiffAlg}}
\newcommand{\saft}[1]{{#1}^\bullet}

\newcommand{\cl}[1]{\bar{{#1}}}

\newcommand{\Circ}[1]{{#1}^{\circ}}

\newcommand{\Der}[2]{\mathrm{Der}_{\Sscript{#1}}(#2)}

\newenvironment{Array}[2][1]
  {\def\arraystretch{#1}\begin{array}{#2}}
  {\end{array}}

\newcommand{\Uv}{\Circ{U}_{\Sscript{(M^*)}}}

\newcommand{\N}{\mathbb{N}}
\newcommand{\K}{\Bbbk}
\newcommand{\C}{\mathbb{C}}
\newcommand{\R}{\mathbb{R}}

\newcommand{\KK}{\mathbb{K}}
\newcommand{\FF}{\mathbb{F}}
\newcommand{\ZZ}{\mathbb{Z}}
\newcommand{\QQ}{\mathbb{Q}}

\newcommand{\sS}{\mathscr{S}}

\newcommand{\cA}{{\mathcal A}}

\newcommand{\cD}{{\mathcal D}}

\newcommand{\cF}{{\mathcal F}}

\newcommand{\cH}{{\mathcal H}}

\newcommand{\cL}{{\mathcal L}}

\newcommand{\cN}{{\mathcal N}}

\newcommand{\cP}{{\mathcal P}}

\newcommand{\cU}{{\mathcal U}}

\makeatletter
\DeclareMathSizes{12}{12}{6}{5}
\makeatother
\newcommand{\Sscript}[1]{{#1}}

\makeatletter																																																							
\def\namedlabel#1#2{\begingroup
    #2%
    \def\@currentlabel{#2}%
    \phantomsection\label{#1}\endgroup
}
\makeatother																																																			   

\excludeversion{invisible}
\allowdisplaybreaks

\title[The Hopf algebroid structure of differentially recursive sequences]{The Hopf algebroid structure of differentially recursive sequences}

\author{Laiachi El Kaoutit}
\address{Universidad de Granada. Departamento de \'{A}lgebra and IEMath. Facultad de Ciencias. Fuente Nueva s/n. E18071 Granada, Spain}
\email{kaoutit@ugr.es}
\urladdr{http://www.ugr.es/~kaoutit/}

\author{Paolo Saracco}
\address{D\'epartement de Math\'ematique, Universit\'e Libre de Bruxelles, Boulevard du Triomphe, B-1050 Brussels, Belgium.}
\email{paolo.saracco@ulb.ac.be}
\urladdr{sites.google.com/view/paolo-saracco}

\date{\today}
\subjclass[2010]{Primary  12H05, 16S32, 16T05, 34M15; Secondary 05A19, 03D20, 34G10; 41A58} 
\keywords{Differential fields; Linear differential matrix equations; Recursive sequences; Series expansions; Rings of differential operators;  Commutative and Co-commutative Hopf algebroids; Hurwitz series; Taylor map.  }
\thanks{This paper was written while P. Saracco was member of the ``National Group for Algebraic and Geometric Structures and their Applications'' (GNSAGA-INdAM). He acknowledges FNRS support through a collaborateur scientifique position (project ``(CO)REPRESENTATIONS'', application number 34777346). Research supported by the Spanish Ministerio de Econom\'{\i}a y Competitividad  and the European Union FEDER, grant MTM2016-77033-P}

\begin{document}

\begin{abstract}
A differentially recursive sequence over a differential field is a sequence of elements satisfying a homogeneous differential equation with non-constant coefficients (namely, Taylor expansions of elements of the field) in the differential algebra of Hurwitz series.
The main aim of this paper is to explore the space of all differentially recursive sequences over a given field with a non-zero differential. We show that these sequences form a two-sided vector space that admits, in a canonical way, a structure of Hopf algebroid over the subfield of constant elements. We prove that it is the direct limit, as a left comodule, of all spaces of formal solutions of linear differential equations and that it satisfies, as Hopf algebroid, an additional universal property. When the differential on the base field is zero, we recover the Hopf algebra structure of linearly recursive sequences.  
\end{abstract}

\maketitle

\tableofcontents

\pagestyle{headings}

\section{Introduction}

This section represents a self-contained introduction to this note. After giving a little background on linear differential equations over a differential field and on the study of their solutions from the point of view of Hurwitz series, we clarify our motivations in reconsidering this subject. The main results of the paper are reported herein as well, in great detail. The last part is devoted to introduce the essential notions and notations that will be employed all over the text. 

\subsection{Motivation and overview}
Let $(\KK,\partial)$ be a differential field with subfield of constants 
$$
\Bbbk:=\KK^{\Sscript{\partial}}=\big\{ c \in \KK|\,\, \partial(c)=0 \big\}.
$$ 
Assume that we are given a homogeneous linear scalar differential equation of order $n$, with not necessarily constant coefficients, of the form
\begin{equation}\label{eq:diff}
\partial^d\left(y\right) - \left(c_{d-1}\partial^{d-1}\left(y\right) + \cdots + c_{1}\partial\left(y\right) + c_0y\right) = 0
\end{equation}
where $c_i\in\KK$ for $i=1,\ldots,d$. In some very particular circumstances, $\KK$ already contains a full set of solutions (\ie $d$ linearly independent solutions) of \eqref{eq:diff}, such as it happens for the equation $\partial\left(y\right)-y/z = 0$ over $\C(z)$, but in general this is not the case (for instance the equation $\partial\left(y\right)-cy/z=0$ over $\C(z)$ admits solutions in $\C(z)$ if and only if $c$ is an integer). One then looks for differential field (or even differential ring) extensions of $\KK$ containing the missing solutions. This is, in fact, one of the objectives of what is known in the literature as \emph{differential Galois theory} \cite{PutSinger}. 

In a series of papers \cite{Keigher-huw,Keigher-Pritch,Keigher-diff}, Keigher and collaborators studied Hurwitz series as a practical way of formally integrating homogeneous linear differential equations over fields (or, more generally, rings, possibly with zero differential). Namely, one introduces the differential algebra of Hurwitz series $(\cH(\KK),\cN)$ over $\KK$ as $\cH(\KK) = \KK^\N$ with the Hurwitz product 
\begin{equation}\label{Eq:Hurwitz}
\alpha \cdot \beta \,\,=\,\, \Big(\sum_{0\leq k \leq n}\binom{n}{k} \alpha_k\beta_{n-k}\Big)_{n \, \in \, \mathbb{N}}, \quad \forall\, \alpha, \beta\, \in \,  \cH(\KK),
\end{equation}
and the differential map is the shift operator given by 
\begin{equation}\label{eq:N}
\cN:\cH(\KK)\longrightarrow  \cH(\KK),\quad \Big(  \left(\alpha_0,\alpha_1,\alpha_2,\ldots\right)\longmapsto \left(\alpha_{1},\alpha_{2},\alpha_{3},\ldots\right) \Big).
\end{equation} 
In this way, one may look at \eqref{eq:diff} as a differential equation over $\cH(\KK)$ of the form:
\begin{equation}\label{eq:diffnaive}
\cN^{d} - \left(s\left(c_{d-1}\right)\cN^{d-1} + \cdots + s\left(c_{1}\right)\cN + s\left(c_0\right)\right) = 0,
\end{equation}
where $s:\KK\to\cH(\KK),\, x\mapsto \left(x,0,0,\ldots\right)$, is called \emph{the source map}. Since $\cN(\alpha)=0$ if and only if  $\alpha=s(c)$ for some $c\in \KK$, $\KK$ can be identified with the subalgebra of constant elements of the differential $\Bbbk$-algebra $(\cH(\KK), \cN)$ via $s$.
As a consequence, equation \eqref{eq:diffnaive} is now a differential equation with constant coefficients and one can solve it to find the so called \emph{formal solutions} of \eqref{eq:diff} in $\cH(\KK)$, that is, sequences $\alpha\in\KK^\N$ such that
\begin{equation}\label{eq:Bistro}
\alpha_{n+d} = c_1\alpha_{n+d-1} + \cdots + c_d\alpha_{n}, \quad \text{ for all  } \, n\geq 0.
\end{equation}
It is noteworthy that $\alpha\in\cH(\KK)$ is a formal solution of \eqref{eq:diff} in this sense if and only if it is a \emph{linearly recursive sequence} over $\KK$, that is, it satisfies a recurrence like \eqref{eq:Bistro}. For the sake of comparison with what follows, let us highlight that in order to write \eqref{eq:Bistro} at level $n+d$, one needs to place the coefficients $(c_{1}, \cdots, c_{d})$ along the vector $(\alpha_{n+d}, \cdots, \alpha_{n+1})$ in a linear way, without deriving the $c_{i}$'s.

Linearly recursive sequences arise widely in mathematics and have been studied extensively and from different points of view. See for example \cite{FutiaMullerTaft,petersontaft,taft} concerning their connections with Hopf algebras and the Sweedler dual $\Circ{\KK[X]}$ of the coordinate algebra of the affine additive group $\left(\KK,+,0\right)$ and \cite{LaiachiPaolo} concerning their topological structure and properties. For a  survey on the topic, we refer to \cite{poorten}. 

Despite the strong motivations supporting Keigher's approach, that is to say, the fact that the ``universal'' space of formal solutions admits the structure of a Hopf algebra,  we believe that the argument reported above has a disadvantage: the inclusion $s:\KK\to\cH(\KK)$ does not make of $\cH(\KK)$ a differential extension of $\KK$, as it does not commute with the differentials. Therefore, we could not see how to relate Keigher's formal solutions to solutions of the original equation \eqref{eq:diff} over $\KK$. In particular it is not clear, at least to us, how to relate the Hopf algebra of formal solutions with the affine algebraic $\Bbbk$-group attached to the initial equation \eqref{eq:diff}, neither how to construct the Picard-Vessiot extension of this equation out of this Hopf algebra.

In the present paper, our aim is to overcome the aforementioned obstacle and to offer to the reader a genuine differential extension of $\KK$ containing all (formal) solutions to homogeneous linear differential equations over $\KK$. Namely, our approach in studying \eqref{eq:diff} by means of the differential algebra $(\cH(\KK), \cN)$ will take into account the injective ``Taylor map'': $t:\KK\to\cH(\KK),\, x\mapsto \left(x,\partial(x),\partial^2(x),\ldots\right)$, referred to as the \emph{the target map}. In this way, we are able to show that $\cH(\KK)$ contains a distinguished two-sided vector space, the one of all \emph{differentially recursive sequences} hereby introduced, that naturally carries a commutative Hopf algebroid structure (as linearly recursive sequences were carrying a Hopf algebra structure). We will show how it can be realized as the universal object satisfying \emph{two} universal properties, and how, as left comodule, it turns out to be the direct limit of the spaces of formal solutions of linear differential equations (compare with \cite{Keigher-diff}).

\subsection{Description of the main results}
We consider $\cH(\KK)$ as a differential extension of $\KK$ via the target map above  $t:\KK\to\cH(\KK)$. By identifying $\KK$ with its image via $t$, we may now look at \eqref{eq:diff} as an equation over $\cH(\KK)$ of the form
\begin{equation}\label{eq:diffseria}
\cN^{d} - \left(t\left(c_{d-1}\right)\cN^{d-1} + \cdots + t\left(c_{1}\right)\cN + t\left(c_0\right)\right) = 0.
\end{equation}
A sequence $\alpha\in\KK^\N$ is a solution of \eqref{eq:diffseria} if and only if it satisfies a recursive relation of the form
\begin{gather}
\alpha_{n+d}  \,=\,  \sum_{k=0}^{n}\binom{n}{k} \partial^{k}(c_{d-1}) \alpha_{n+d-k-1} + \cdots +   \sum_{k=0}^{n}\binom{n}{k} \partial^{k}(c_{1}) \alpha_{n-k+1} +  \sum_{k=0}^{n}\binom{n}{k} \partial^{k}(c_{0}) \alpha_{n-k} \label{eq:DRS}\\
= c_{d-1}\alpha_{n+d-1} + n\partial(c_{d-1})\alpha_{n+d-2} + \binom{n}{2} \partial^2(c_{d_1})\alpha_{n+d-3} + \cdots + \partial^n(c_{d-1})\alpha_{d-1} + c_{d-2}\alpha_{n+d-2} +  \cdots + \partial^n(c_0)\alpha_{0}, \notag
\end{gather}
for $c_i\in\KK$. For instance, if $d=2$ and equation \eqref{eq:diffseria} has the form $\cN^{2}- t\left(c_{1}\right)\cN -  t\left(c_0\right)=0$, then the attached differential recursive relation can be written as follows:
$$
\alpha_{n+2} \,\, =\,\, \binom{n}{0} \partial^{0}c_{1}\alpha_{n+1} \, +\, \sum_{k=0}^{n-1}\Bigg( \binom{n}{k+1} \partial^{k+1}c_{1} + \binom{n}{k}\partial^{k}c_{0} \Bigg)\,\alpha_{n-k} \,+\, \binom{n}{n} \partial^{n}c_{0}\alpha_{0}, \quad \forall \,  n \geq 0.
$$ 
A sequence satisfying a relation of the form \eqref{eq:DRS} will be referred to as a \emph{differentially recursive sequence}. We will prove that the collection $\dlin{\KK}$ of all differentially recursive sequences over $\KK$ is not only a differential extension of $\KK$ providing all solutions to linear differential equations, but it also enjoys a structure of a commutative Hopf algebroid over $\Bbbk$ (\ie that of an affine groupoid $\Bbbk$-scheme) converting it into a universal construction (in the category theoretical sense) in two ways. On the one hand, it is the universal object $\Circ{\KK[\partial]}$ provided by the Tannaka-Kre\u{\i}n reconstruction procedure applied to the forgetful functor from the category of all differential $\KK$-modules (\ie modules over the \emph{ring of linear differential operators} $\KK[\partial]\coloneqq \KK[Y, \partial]$, constructed as the Ore extension of $\KK$ via the derivation $\partial$, with underlying  finite-dimensional vector space structure) to the category of finite-dimensional right $\KK$-vector spaces, as in \cite{LaiachiGomez}. On the other hand, it is the ``biggest'' $\KK$-coring $\saft{\KK[\partial]}$ inside the right $\KK$-linear dual $\KK[\partial]^* = \Hom{}{}{\KK}{}{\KK[\partial]}{\KK}$  (see subsection \ref{ssec:I3} below for general notations) provided by the ring/coring duality via the Special Adjoint Functor Theorem, as in \cite{ArdiLaiachiPaolo}. The following resumes our main achievements.

\begin{theorem*}
The $\K$-algebra $\dlin{\KK}$ of differentially recursive sequences (i.e., satisfying \eqref{eq:DRS}) enjoys a structure of commutative Hopf algebroid with base $\Bbbk$-algebra $\KK$. The structure maps are explicitly given by the source $s:\KK\to\dlin{\KK},\,x\mapsto (x,0,\ldots)$, the target $t:\KK\to\dlin{\KK},\,x\mapsto \left(\partial^n(x)\right)_{n\geq 0}$, the counit $\varepsilon:\dlin{\KK} \to \KK, \alpha\mapsto \alpha(0)$, the comultiplication
\[
\Delta:\dlin{\KK} \longrightarrow \dlin{\KK}\tensor{\KK}\dlin{\KK}, \qquad \alpha\longmapsto \sum_{i=0}^d \cN^i\left(\alpha\right) \tensor{\KK} \left(y_i^*\left(y_n\right)\right)_{n\geq 0}
\]
\Big(where $\left\{ y_i^* \right\}_{0\leq i \leq d}$ is a suitable dual basis of the space of formal solutions of the differential equation satisfied by $\alpha$ and $y_n$ are the canonical images of the operators $\partial^n$ therein\Big) and the antipode
\[
\sS:\dlin{\KK}\longrightarrow \dlin{\KK}, \qquad \alpha \longmapsto \left(\sum_{k=0}^n\binom{n}{k}(-1)^{n-k}\partial^k\left(\alpha(n-k)\right)\right)_{n\geq 0}.
\]
It turns out that, with this structure, there is a chain of isomorphisms  
$$
\Circ{\KK[\partial]} \, \cong \, \dlin{\KK} \, \cong \, \saft{\KK[\partial]}
$$ of commutative Hopf algebroids (see \cite{ArdiLaiachiPaolo} for the definitions of both $(-)^{\circ}, \saft{(-)}$). Furthermore, if by $M_{\Sscript{\cL}}^*$ we denote the (formal) solution space of the differential equation $\cL(y)=0$, then the family $\left\{ M_{\Sscript{\cL}}^* \right\}_{\Sscript{\cL}}$ forms a directed system of  left $\dlin{\KK}$-comodules such that  we have an isomorphism
\[
\injlimit{\cL}{M_{\Sscript{\cL}}^*} \cong \dlin{\KK}
\]
of left comodules.
\end{theorem*}

The relation of differentially recursive sequences with  linear recursive sequences, is given in form of the following commutative diagram of $\K$-algebras, with injective arrows: 
\begin{equation*}
\begin{gathered}
\xymatrix@=15pt{
 & \cH(\KK) & \\
\dlin{\KK} \ar@{^{(}->}[ur] & & \Lin{\KK} \ar@{_{(}->}[ul] \\
 & \Lin{\K} \ar@{_{(}->}[ul] \ar@{^{(}->}[ur] & 
}
\end{gathered}
\end{equation*}
where $\Lin{\FF}$ denotes the $\FF$-vector space of linearly recursive sequences for a given field $\FF$ (with or without differential). We provide concrete examples (see the ones reported in Example \ref{exam:Dlink} below) to show that the images of $\dlin{\KK}$ and $\Lin{\KK}$ inside $\cH(\KK)$ do not coincide. 
We conclude the paper by giving a brief comment on the Picard-Vessiot ring extension of equation \eqref{eq:diff} and its relation with the Hopf algebroid $\dlin{\KK}$.

\subsection{Notation and basic notions}\label{ssec:I3}

For $R,S,T$ three rings, $N$ an $(S,T)$-bimodule, $M$ an $(R,S)$-bimodule and $P$ an $(R,T)$-bimodule, the hom-tensor adjunction states that we have bijective correspondences
\[
\Hom{S}{}{T}{}{N}{ \Hom{R}{}{}{}{M}{P} } \cong \Hom{R}{}{T}{}{M\tensor{S}N}{P} \cong \Hom{R}{}{S}{}{M}{\Hom{}{}{T}{}{N}{P}}.
\]
The $(S,T)$-bimodule structure on $\Hom{R}{}{}{}{M}{P}$ (and, similarly, the $(R,S)$-bimodule one on $\Hom{}{}{T}{}{N}{P}$) is given as follows. For all $s\in S$, $t\in T$, $f\in \Hom{R}{}{}{}{M}{P}$ and $m\in M$,
\[
\left(s\cdot f \cdot t\right)(m) = f\left(m\cdot s\right)\cdot t.
\]
Since every $s\in S$ (and, analogously, every $t\in T$) induces an $R$-linear endomorphism $\rho_s: M\to M, m\mapsto m\cdot s$, we may equivalently write
\[
s\cdot f \cdot t = \rho_t \circ f \circ \rho_s.
\]
Similarly, every $r\in R$ (and every $s\in S$) induces the $T$-linear morphism $\lambda_r : N \to N, n\mapsto r\cdot n$, and hence
\[
r\cdot g \cdot s = \lambda_s \circ g \circ \lambda_r
\]
for all $g\in \Hom{}{}{T}{}{N}{P}$. We will often omit the $\cdot$ symbol in what follows.

By a differential algebra over a commutative ring $R$ we mean an $R$-algebra $A$ together with an $R$-linear endomorphism $\partial_A:A\to A$ satisfying the Leibniz rule 
$$
\partial_A(ab) = \partial_A(a)b+ a \partial_A(b),
$$ 
for all $a,b\in A$. The following relation can be proven inductively
\begin{equation}\label{eq:derprod}
\partial_A^n(ab) = \sum_{k=0}^n \binom{n}{k}\partial_A^k(a)\partial_A^{n-k}(b).
\end{equation}
With $\left(\KK,\partial\right)$ we always denote a differential field with subfield of constants $\K$. The unadorned tensor product $\otimes$ is that over $\K$. If another different field will be required, we denote it by $\FF$.


\section{Differential operators and differentially recursive sequences over differential fields}

This section is devoted to stating and proving the main result of the paper, concerning the Hopf algebroid structure on the space of differentially recursive sequences. Namely, after recalling explicitly the Hopf algebroid structure of the differential operator algebra of a given differential field, we introduce the space of \emph{differentially recursive sequences}, we prove that it inherits a Hopf algebroid structure from that and, finally, we compare these  sequences with the usual linearly recursive ones. 

\subsection{Differential operators over differential fields and Hopf algebroids}
Given $\KK,\partial$ and $\K$ as before,  we consider the skew polynomial algebra (also known as Ore extension) $U_{\KK}\coloneqq \KK[Y;\partial]$ associated with $\KK$. This is the free  $\K$-algebra generated by $\KK$ and an element $Y$ subject to the relations
\[
xY = Yx + \partial (x)
\]
for all $x\in\KK$. Over $\KK$, it coincides with the right $\KK$-vector space generated by the symbols $\{Y^i\mid i\geq 0\}$: $\bigoplus_{i\geq 0}Y^i\KK$. Therefore, a generic element will be written as $\sum_{i=0}^d Y^ic_i$ with $c_i\in\KK$ for $i=0,\ldots,d$.

For the sake of simplicity, we will set $U\coloneqq U_{\KK}$ when the field $\KK$ is clear from the context.

\begin{example}\label{ex:easy}
The field $\C$ of complex numbers (and, in general, any field) with the zero derivation $\partial\equiv 0$ is a differential field with constant field $\C$ itself. It's associated skew polynomial algebra is the algebra of polynomials in one indeterminate $\C[Y]$.
\end{example}

\begin{example}\label{ex:diffop}
For a field $\FF$ of characteristic $0$, the field of rational functions $\FF(X)=\left\{p(X)/q(X)\mid q(X)\neq 0\right\}$ in one indeterminate, with derivation $\partial_X$ uniquely extended from $\partial_X:\FF[X]\to \FF[X]$, is a differential field with constant field $\FF$ itself. Notice that
\[
\partial_X\left(\frac{1}{q(X)}\right) = - \frac{\partial_X(q(X))}{q(X)^2}
\]
for all non-zero $q(X)\in \FF[X]$, in light of the Leibniz rule. In this case, the associated skew polynomial algebra is the algebra of differential operators of $\FF\left(X\right)$ (see \cite[Corollary 15.2.5 and Theorem 15.5.5]{McConnelRobson}).
\end{example}

In general, the skew polynomial algebra $U$ can be considered as an algebra of differential operators of $\KK$ (\ie a subalgebra of the algebra of differential operators of $\KK$).
It is well-known that for $\C$ a differential field as in Example \ref{ex:easy}, the skew polynomial algebra $U=\C[Y]$ is a Hopf algebra and its finite dual Hopf algebra $\Circ{U}$ coincides (up to isomorphism) with the algebra $\Lin{\C}$ of linearly recursive sequences over $\C$ (see \cite[\S 3.5]{Abe} for the explicit definition of the finite dual Hopf algebra). This observation played a fundamental role in \cite{LaiachiPaolo} and allowed us to reveal the rich topological structure of $\Lin{\C}$. The present section is devoted to show that a similar identification holds over more general differential fields.

\begin{remark}\label{rem:CHAlgd}
For any differential field $\left(\KK,\partial\right)$ with field of constants $\K$, the associated Ore extension $U=\KK[Y;\partial]$ is a (right) cocommutative Hopf algebroid over $\KK$. We refer to \cite[\S2.2]{ArdiLaiachiPaolo} for the technical details of the definition of a cocommutative Hopf algebroid we will use in this paper. Here we just recall briefly the structure maps and their properties for the case we are interested in. In details, identify $xY^0$ with $x$ for every $x\in \KK$, so that we may consider the assignment $\tau: \KK\to U,\, x\mapsto x$. This is a morphism of $\K$-algebras which converts $U$ into a $\KK$-ring (not a $\KK$-algebra, as $\tau$ does not land into the center of $U$). As a consequence, $U$ is a two-sided $\KK$-vector space with actions
\begin{equation}\label{eq:useful}
Y^n\cdot x = Y^nx \qquad \text{and} \qquad x\cdot Y^n = \sum_{k=0}^n \binom{n}{k}Y^k\partial^{n-k}(x),
\end{equation}
for all $n\geq 0$, $x\in \KK$. The counit is provided by the assignment $\varepsilon:U\to \KK,\, Y^n\mapsto \delta_{0,n},$ extended by right $\KK$-linearity. The unique right $\KK$-linear morphism $\Delta:U\to U\tensor{\KK}U$ satisfying 
\[
\Delta\left(Y^n\right)= \sum_{k=0}^n \binom{n}{k} Y^k \tensor{\KK} Y^{n-k}
\]
for all $n\geq 0$ (where the tensor product $\tensor{\KK}$ is taken by considering $U$ as a symmetric $\KK$-bimodule with left action induced by the right one) is a well-defined coassociative and counital comultiplication that lands into the so-called Sweedler-Takeuchi $\times$-product (\cite{Sweedler-IHES, Takeuchi:77}) and such that $\Delta:U\to U\times_{\KK}U$ is a $\K$-algebra map. The translation map is provided by
\[
\beta^{-1}(1\tensor{\KK} Y^{n}) = \sum_{k=0}^n \binom{n}{k}(-1)^kY^k \tensor{\KK} Y^{n-k}.
\]
Therefore, the space $U^*\coloneqq \Hom{}{}{\KK}{}{U}{\KK}$ of right $\KK$-linear morphisms from $U$ to $\KK$ becomes a ring with the convolution product
\begin{equation}\label{eq:conv}
\left(f*g\right)(u) = \sum_{(u)} f\left(u_{(1)}\right)g\left(u_{(2)}\right),
\end{equation}
for all $f,g\in U^*$ and $u\in U$ and where $\Delta(u) =\sum_{(u)}u_{(1)}\tensor{\KK}u_{(2)}$ by resorting to the so-called Sweedler Sigma Notation. The interested reader may check that, since the filtration $F^n\left(\KK[Y;\partial]\right)\coloneqq \bigoplus_{k=0}^nY^k\KK$ is an admissible filtration on $U$ (in the sense of \cite[\S3.4]{LaiachiPaolo-big}) and since the translation map is filtered with respect to this filtration, $U^*$ becomes in fact a complete commutative Hopf algebroid in the sense of \cite{LaiachiPaolo-big} (see in particular \cite[Proposition 3.16]{LaiachiPaolo-big}).
\end{remark}

In what follows, we will denote by $\Circ{U}$ the finite dual Hopf algebroid of $U$ as constructed in \cite{LaiachiGomez} and not its Sweedler dual (see also \cite{ArdiLaiachiPaolo}).

\begin{remark}
The finite dual Hopf algebroid $\Circ{U}$ is not, in general, the same as the Sweedler dual of $U$. In fact, recall that the Sweedler dual of an algebra $A$ is the space of linear functionals on $A$ vanishing on a finite-codimensional two-sided ideal. If we consider the example $\KK=\C(X)$, then $U = \C(X)[Y;\partial]$ is a simple ring (in light of \cite[Theorem 1.8.4]{McConnelRobson}, for instance). Therefore, the Sweedler dual of $U$ is zero, while the finite dual in the sense of \cite{LaiachiGomez} is not. 
\end{remark}

\begin{example}
For $\KK = \C(X)$, $U$ is the universal enveloping Hopf algebroid of the Lie-Rinehart algebra $L\coloneqq \Der{\K}{\KK}=\C(X)\partial_X$, the one-dimesional vector space generated by $\partial_{X}$.
\end{example}

\begin{remark}\label{rem:action*}
Notice that $\KK$ is a (right) $U$-module with action uniquely determined by $x\cdot Y = \partial(x)$ for all $x\in \KK$. By the hom-tensor adjunction, this induces a $\K$-linear map $\mu:\KK\to {U}^*=\Hom{}{}{\KK}{}{U}{\KK}$ such that $\mu(x)(u)=x\cdot u$ for all $x\in \KK, u\in  U$. This $\mu$ turns out to be a $\K$-algebra morphism, since for all $n\geq 0$
\[
\mu(xy)\left(Y^n\right) = \partial^n(xy) \stackrel{\eqref{eq:derprod}}{=} \sum_{k=0}^n \binom{n}{k} \partial^k(x)\partial^{n-k}(y) \stackrel{\eqref{eq:conv}}{=} (\mu(x)*\mu(y))\left( Y^n\right).
\]
As a consequence, $U^*$ is a $\left(\KK\otimes\KK\right)$-algebra with $\eta':\KK\otimes\KK\to U^*$ uniquely determined by $\eta'(x\otimes 1) = x\varepsilon$ and $\eta'(1\otimes x)=\mu(x)$ for all $x\in \KK$. It is also a right $U$-module with action induced by left multiplication on $U$:
\begin{equation}\label{eq:actf}
\begin{gathered}
\xymatrix@R=0pt@C=3pt{ f\triangleleft u = f\circ\lambda_u:  & U  \ar@{->}[rrr] &&& \KK \\ & v \ar@{|->}[rrr] &&& f(uv)}
\end{gathered}
\end{equation}
 for all $f\in U^*$ and $u\in U$ and where $\lambda_u(v) = uv$ for all $v\in U$.
\end{remark}

\subsection{Differentially recursive sequences}

The set $\KK^{\N}$ of all denumerable sequences of elements in $\KK$ admits a left (eventually, symmetric) $\KK$-vector space structure induced by the equivalent description
\[
\KK^{\N}=\mathsf{Fun}(\N,\KK)=\Big\{\alpha:\N\to \KK\Big\},
\]
that is to say, given by the componentwise sum and action:
\[
\left(\alpha+\beta\right)(n)=\alpha(n)+\beta(n) \qquad \text{and} \qquad \left(x\alpha\right)(n)=x\alpha(n)
\]
for all $\alpha,\beta\in\KK^\N$, $x\in\KK$ and $n\in \N$. As a matter of notation, a sequence in $\KK^\N$ will be denoted either as $\left(\alpha_n\right)_{n\geq 0}$, or as $\left(\alpha(n)\right)_{n\geq 0}$ or simply as $\alpha$, meaning by this the function $\alpha:\N\to\KK$. On $\KK^{\N}$ we can also consider a product, called the \emph{Hurwitz product},
\begin{equation}\label{eq:Hurwitz}
(\alpha\cdot\beta)(n)=\sum_{k=0}^n \binom{n}{k}\alpha(k)\beta(n-k)
\end{equation}
for all $\alpha,\beta\in\KK^\N$ and $n\in \N$, and the \emph{shift operator} $\cN:\KK^{\N}\to \KK^{\N}$, given by $\cN(\alpha)(n)=\alpha(n+1)$ for all $\alpha\in\KK^\N$, $n\in\N$. Denote by $\cH(\KK)$ the ring $\KK^\N$ with the Hurwitz product. With respect to this structure, $\cH(\KK)$ becomes a commutative $\KK$-algebra with unit morphism $s:\KK\to \KK^{\N},\,x\mapsto (x,0,\ldots),$ and $\cN$, which is already a $\KK$-linear endomorphism, becomes a derivation. Since this structure will be fixed throughout the paper, we will often omit the $\cdot$ symbol.

\begin{remark}
Observe that $\cN(\alpha) = 0$ if and only if $\alpha(n) = x\delta_{n,0}$ for some $x\in\KK$, whence $\cH(\KK)$ is a differential algebra with subalgebra of constants $\KK$. In particular, 
\[
\cN\in\Der{\KK}{\cH(\KK)}=\Big\{\delta\in \End{\K}{\cH(\KK)} \mid \delta(\alpha\beta) = \delta(\alpha)\beta + \alpha \delta(\beta) \text{ and } \delta \circ s= 0\Big\}=:\mathsf{Der}^s_{\K}\left(\cH(\KK)\right).
\]
\end{remark}

In addition, $\cH(\KK)$ is a $\KK$-algebra with respect to $t:\KK\to\cH(\KK),\,x\mapsto \left(\partial^n(x)\right)_{n\geq 0}$. In the literature, the map $t$ has been called the \emph{Hurwitz mapping} of $\partial$ (see \cite{Keigher-huw}) and the element $t(x)$ has been called the \emph{Hurwitz expansion} of $x$. 
In particular, $\cH(\KK)$ becomes a commutative $\left(\KK\otimes\KK\right)$-algebra with unit
\[
\eta: \left(\KK\otimes\KK\right) \to \cH(\KK) ,\quad x\otimes y \mapsto s(x)t(y)=\left(x\partial^n(y)\right)_{n\in \N}.
\]
As such, we will consider it as a two-sided $\KK$-vector space with left $\KK$-action induced by $s$ and right $\KK$-action induced by $t$. It is noteworthy that $t:\left(\KK,\partial\right)\to \left(\cH(\KK),\cN\right)$ is a morphism of differential $\K$-algebras. It is also augmented over $\KK$ via the algebra map $\varepsilon : \cH(\KK) \to \KK$ sending $\alpha \mapsto \alpha_0$.

\begin{remark}\label{rem:actionseq}
Observe that $\End{\K}{\cH(\KK)}$ is naturally endowed with a two-sided $\KK$-vector space structure given by
\[
x\cdot \cF \cdot y \coloneqq  s(x)t(y)\cF : \alpha \mapsto s(x)t(y)\cF(\alpha)
\]
for all $x,y\in\KK$, $\cF\in \End{\K}{\cH(\KK)}$ and $\alpha\in \cH(\KK)$. Thus, $\cH(\KK)$ is naturally a right $U$-module, infinite-dimensional over $\KK$, with action uniquely determined by
\begin{equation}\label{eq:actalpha}
\alpha\triangleleft Y \coloneqq  \cN(\alpha)
\end{equation}
for all $\alpha\in\cH(\KK)$.
\end{remark}

We are now ready to introduce differentially recursive sequences.

\begin{definition}\label{def:LRS}
An operator $\cL=\sum_{i=0}^dt(c_i)\cN^i\in \End{\KK}{\cH(\KK)}$, with $c_i\in\KK$ and $c_d\neq 0$, is said to have \emph{order} (or \emph{degree}) $d$. We set 
$$
\dlin{\KK}\coloneqq \left\{\alpha\in\cH(\KK)\ \left| \ \cL\left(\alpha\right) = 0 \text{ for some }\cL=\sum_{i=0}^dt(c_i)\cN^i, c_i\in\KK \right. \right\}. 
$$
An element $\alpha$ of $\dlin{\KK}$ is called a \emph{differentially recursive sequence} over $\KK$. The minimum $d$ such that $\alpha$ is annihilated by an operator of order $d$ is said to be of \emph{order} $\alpha$.
\end{definition}

\begin{remark}\label{rem:terminolology}
Very informally speaking, the motivation for this terminology is twofold. On the one hand, a sequence as in Definition \ref{def:LRS} satisfies a differential recursive relation with differential coefficients in the same way a linearly recursive one satisfies a linear relation with constant coefficients (with respect to the  differential  $\cN$). On the other hand, as we will see in Remarks \ref{rem:polykill} and \ref{rem:polykilllin}, a sequence is differentially recursive if and only if it is killed by a differential polynomial while it is linearly recursive if and only if it killed by an ordinary one.
\end{remark}

\begin{example}
If $\KK=\C$ with the zero derivation, then $s=t$ and a differentially recursive sequence in the sense of Definition \ref{def:LRS} is the same as a linearly recursive sequence in the classical sense, that is to say, a sequence of elements of $\C$ which satisfies a recurrence relation with constant coefficients. Indeed, $\alpha\in\dlin{\C}$ if and only if there exists $c_0,\ldots,c_d\in\C$, $c_d=1$, such that
\[
0=\sum_{i=0}^dt(c_i)\cN^i(\alpha) = \left(\sum_{i=0}^dc_i\alpha(n+i)\right)_{n\geq 0}
\]
if and only if  $\alpha_{n+d} = -\left(c_{d-1}\alpha_{n+d-1} + c_{d-2}\alpha_{n+d-2} + \cdots + c_0\alpha_{n}\right)$, for every $n\geq 0$.
\end{example}

\begin{example}
If $\KK=\C(z)$ with derivation $\partial=\partial/\partial z$, then $t\left(p(z)/q(z)\right) = \left(\partial^n\left(p(z)/q(z)\right)\right)_{n\geq 0}$. Consider the particular case of an operator $\cL=\sum_{i=0}^dt(a_i)\cN^i$ in which $a_i\in\C$ for all $i=0,\ldots,d$. Then the elements $f\in\C(z)$ for which $\cL(t(f))=0$ are exactly the solutions to the differential equation 
\[
a_0y + a_1\partial(y) + a_2\partial^2(y) + \cdots + a_d\partial^d(y) = 0
\]
with constant coefficients in $\C(z)$.
\end{example}


\subsection{The Hopf algebroid structure on differentially recursive sequences}

Henceforth, for the sake of simplicity and clearness, a sequence $\alpha=\left(\alpha(n)\right)_{n\geq 0}=\left(\alpha_n\right)_{n\geq 0}$ will be denoted also by $\left(\alpha_\bullet\right)$.

\begin{proposition}\label{prop:iso}
We have an isomorphism of $\left(\KK\otimes\KK\right)$-algebras 
\[
\Phi:U^*\to \cH(\KK) : f \mapsto \left(f\left(Y^\bullet\right)\right)
\]
with inverse sending the sequence $\alpha$ to the right $\KK$-linear morphism $f_{\alpha}$ uniquely determined by the assignment $Y^n\mapsto \alpha(n)$ for all $n\in\N$. Moreover, $\Phi$ is also right $U$-linear with respect to the actions \eqref{eq:actf} and \eqref{eq:actalpha}, namely
\begin{equation}\label{eq:PhiLin}
\Phi\left(f\right)\triangleleft P(Y) = \Phi\left(f\circ\lambda_{P(Y)}\right)
\end{equation}
for every differential polynomial $P(Y)\in U$.
\end{proposition}

\begin{proof}
Let us check explicitly only the last claim. If $P(Y)=\sum_{i=0}^dY^ic_i$, $c_i\in\KK$, then for all $n\geq 0$
\begin{gather*}
\left(\Phi(f)\triangleleft P(Y)\right)(n) = \left(\sum_{i=0}^dt(c_i) \cN^i\left(\Phi(f)\right)\right)(n) = \sum_{i=0}^d\sum_{k=0}^n\binom{n}{k} \partial^k(c_i)f\left(Y^{n-k+i}\right) = f\left(\sum_{i=0}^d\sum_{k=0}^n\binom{n}{k}Y^{n-k+i} \partial^k(c_i)\right) \\
 = f\left(\sum_{i=0}^dY^ic_iY^{n}\right) = \left(f\circ\lambda_{P(Y)}\right)(Y^n) = \Phi\left(f\circ\lambda_{P(Y)}\right)(n). \qedhere
\end{gather*}
\end{proof}

\begin{corollary}
The pair $(\KK, \cH(\KK))$ is a complete Hopf algebroid with respect to the structure maps and the linear topology 
coming from that of $U^*$ via $\Phi$.
\end{corollary}

\begin{proof}
It follows from Proposition \ref{prop:iso} together with the observation of Remark \ref{rem:CHAlgd}.
\end{proof}

\begin{remark}\label{rem:polykill}
Observe that if $P(Y)\coloneqq \sum_{i=0}^dY^ic_i$ then $\cL(\alpha) = \alpha\triangleleft P(Y)$, where $\cL = \sum_{i=0}^dt(c_i)\cN^i$. This in particular justifies the terminology used before: \emph{the degree of $\alpha$}. Furthermore, thanks to the division algorithm on $\KK[Y;\partial]$ (see \cite{Ore}), for every $\alpha\in\dlin{\KK}$ we may assume that $\cL$ such that $\cL(\alpha)=0$ is the operator associated with the monic generator $P(Y)$ of $\ann{\alpha} = \{Q(Y)\in U \mid \alpha \triangleleft Q(Y) = 0\}$.
\end{remark}

\begin{proposition}\label{prop:vanish}
Both maps $s,t:\KK\to\cH(\KK)$ land into $\dlin{\KK}$. Moreover, via $\Phi$ of Proposition \ref{prop:iso}, differentially recursive sequences correspond to linear maps $f:U\to\KK$ vanishing on a principal right ideal of $U$. More precisely, for $f\in U^*$ we have that $\Phi(f)\in \ker{\cL}$ with $\cL=\sum_{i=0}^dt(c_i)\cN^i$, $c_i\in\KK$, if and only if $\ker{f}\supseteq P(Y)U$, the principal right ideal generated by $P(Y) = \sum_{i=0}^dY^ic_i$ in $U$.
\end{proposition}
\begin{proof}
Clearly, for every $x\in\KK$ we have that $\cN(s(x)) = 0$, whence $s(x)$ is differentially recursive. Moreover, since we already know that $\cN(t(x)) = t(\partial(x))$ because $t$ is a morphism of differential algebras, it is clear also that $\cL(t(x))=0$ with $\cL\coloneqq t(\partial(x)) - t(x)\cN$, whence $t(x)$ is differentially recursive as well. This proves the first claim. Concerning the second claim, assume that $\Phi(f)\in \ker{\cL}$ with $\cL=\sum_{i=0}^dt(c_i)\cN^i$ and set $P(Y) \coloneqq  \sum_{i=0}^dY^ic_i\in U$. Then,
\[
0 = \cL(\Phi(f)) = \Phi(f)\triangleleft P(Y) \stackrel{\eqref{eq:PhiLin}}{=} \Phi(f\circ\lambda_{P(Y)}).
\]
However, being $\Phi$ invertible, this means that for every $u\in U$, $f(P(Y)u)=0$ and hence $\ker{f}\supseteq P(Y)U$. Conversely, if $\ker{f}\supseteq P(Y)U$ then $f\circ \lambda_{P(Y)}\equiv 0$ and hence
\[
0 = \Phi(f\circ\lambda_{P(Y)}) \stackrel{\eqref{eq:PhiLin}}{=} \Phi(f)\triangleleft P(Y).
\]
Thus, $\cL\left(\Phi(f)\right)=0$ with operator $\cL=\sum_{i=0}^dt(c_i)\cN^i$.
\end{proof}

As a matter of notation, for every $f\in U^*$ as in Proposition \ref{prop:vanish} we will write $P_f(Y)$ for the polynomial $\sum_{i=0}^dY^ic_i$, so that $\ker{f}\supseteq P_f(Y)U$. If moreover $f = f_\alpha \coloneqq \Phi^{-1}(\alpha)$ then we will write $P_\alpha(Y)$ instead of $P_f(Y)$. In particular, by an harmless abuse of notation, from time to time we will write $\cL=P_\alpha(\cN)$. Observe that for every $x\in\KK$, $P_{t(x)}(Y) = \partial(x) - Yx$.

\begin{remark}\label{rem:Varios}
Let us observe that the foregoing arguments and constructions can be realized over any differential algebra $(A,\partial)$ instead of a differential field $(\KK,\partial)$ with no additional effort. In this more general context one may also show that the assignment $R\mapsto \cH(R)$ induces a functor $\cH:\algk_\K \to \dalgk_\K$ which is right adjoint to the underlying functor $\cU:\dalgk_\K \to \algk_\K$ forgetting the differential structure (see \cite{Keigher-adj}). The unit of this adjunction is exactly $t:(A,\partial) \to (\cH(A),\cN), \, a\mapsto (\partial^\bullet(a)),$ and the counit is $\epsilon:\cH(R) \to R,\,\alpha\mapsto \alpha(0)$. The distinctive feature of the field case is that we may always assume $P_\alpha(Y)$ to be monic and hence the ideal $P_\alpha(Y)U$ to be finite-codimensional, as we will need later on.
\end{remark}

Our next objective is to show that $\dlin{\KK}$ is isomorphic to $\Circ{U}$ as $\left(\KK\otimes\KK\right)$-algebra. Recall from \cite{LaiachiGomez} that $\Circ{U}$ is constructed out of the symmetric rigid monoidal $\K$-linear abelian category of finite-dimensional differential $\KK$-vector spaces. That is to say, out of those finite-dimensional (right) $\KK$-vector spaces $M$ endowed with a $\K$-linear endomorphism $\partial_M:M\to M$ satisfying
\[
\partial_M(mx) = \partial_M(m)x + m\partial(x)
\]
for all $m\in M$, $x\in \KK$. Notice that the following extension of \eqref{eq:derprod} holds
\begin{equation}\label{eq:deract}
\partial_M^n(mx) = \sum_{k=0}^n\binom{n}{k}\partial_M^k(m)\partial^{n-k}(x).
\end{equation}
Henceforth, we resort to the notation of \cite[\S 3.1]{LaiachiGomez}. For a given differential $\KK$-module $M$, we denote by $T_{M}=\End{}{(M,\partial)}$ its $\K$-algebra of differential endomorphisms (e.g., $T_{\KK}=\K$), and for two given differential $\KK$-modules $M,N$, we denote by $T_{MN}$ the $\K$-vector space of all differential morphisms from $(M, \partial_M)$ to $(N,\partial_N)$. The $\left(\KK\otimes\KK\right)$-algebra $\Circ{U}$ is by definition the quotient two-sided $\KK$-vector space:
$$
\Circ{U}\coloneqq  \frac{\underset{\Sscript{(M,\partial)}}{\bigoplus} \,  M^*\tensor{T_M}M}{\langle \psi\tensor{T_N} f x - \psi f \tensor{T_M}x \rangle_{\psi \in N^*,\, x \in M,\, f \in T_{MN} }}
$$
with multiplication (see \cite[Equation (20)]{LaiachiGomez})
\begin{equation}\label{eq:multcirc}
\left( \overline{\varphi \tensor{P} p} \right) \cdot \left( \overline{\psi \tensor{Q} q} \right) = \overline{(\psi \diamond \varphi) \tensor{T_{Q\tensor{\KK}P}} (q \tensor{\KK} p)},
\end{equation}
where $\psi \diamond \varphi : Q\tensor{\KK}P \to \KK,\, q\tensor{\KK}p\mapsto \varphi\left(\psi(q)p\right)$, where the overlined notation stands for the equivalence class of a given element in the direct sum displayed in the numerator.

\begin{remark}\label{rem:hopfalgd}
The $\left(\KK\otimes\KK\right)$-algebra $\Circ{U}$ is, in fact, a commutative Hopf algebroid (see \cite[Theorem 4.2.2]{LaiachiGomez}) with source and target
\[
s_\circ: \KK \to \Circ{U}, \quad k\mapsto \overline{k\tensor{\K}1} \qquad \text{and} \qquad t_\circ:\KK \to \Circ{U}, \quad k\mapsto \overline{1\tensor{\K}k},
\]
comultiplication
\[
\Delta_\circ: \Circ{U} \to \Circ{U}\tensor{\KK}\Circ{U}, \quad \overline{\varphi \tensor{T_M} m } \mapsto \sum_i \overline{\varphi \tensor{T_M} e_i} \tensor{\KK} \overline{e_i^* \tensor{T_M} m} \qquad \Big(\{e_i,e_i^*\}\text{ dual basis of }M\Big),
\]
counit $\varepsilon_\circ:\Circ{U} \to \KK,\, \overline{\varphi \tensor{T_M} m } \mapsto \varphi(m)$, and antipode
\[
S_\circ : \Circ{U} \to \Circ{U}, \qquad \overline{\varphi \tensor{T_M} m } \mapsto \overline{\mathsf{ev}_m\tensor{T_{M^*}} \varphi },
\]
where $\mathsf{ev}_m:M^* \to \KK,\varphi\mapsto \varphi(m)$ is the evaluation at $m$. In addition, it is a differential $\KK$-algebra with respect to the source $s_\circ$ and the derivation
\begin{equation}\label{eq:diffcirc}
\partial_\circ : \Circ{U} \to \Circ{U}, \qquad \overline{\varphi \tensor{T_M} m } \mapsto \overline{\varphi \tensor{T_M} m\triangleleft Y },
\end{equation}
and the target $t_\circ$ becomes a morphism of differential algebras with respect to this structure.
\end{remark}

Furthermore, $\Circ{U}$ always comes with a canonical morphism of $\left(\KK\otimes\KK\right)$-algebras $\zeta:\Circ{U} \to U^*$ which, in this particular situation, turns out to be injective (see \cite[Corollary 3.3.6]{LaiachiGomez}).

Our first step will be that of showing that, for a given differential vector space $(M,\partial_M)$ and a given element $\overline{\varphi\tensor{T_M}m}\in \Circ{U}$ (i.e., the  equivalence class of the homogeneous element $\varphi\tensor{T_M}m \in M^*\tensor{T_M}M$), the sequence $$
\left(\varphi\left(\partial_M^\bullet(m)\right)\right) = \Phi\left(\zeta\left(\overline{\varphi\tensor{T_M}m}\right)\right)\in \KK^{\N}
$$ 
satisfies a particular kind of (differential) recursion.

\begin{proposition}\label{prop:big1}
Given an element of the form $\overline{\varphi\tensor{T_M}m}\in\Circ{U}$, then $\alpha\coloneqq \left(\varphi\left(\partial^\bullet_M(m)\right)\right)$ is a differentially recursive sequence in $\cH(\KK)$. That is to say, there exist $0 \leq d\leq \dim_{\KK}(M)$ and $c_0,\ldots,c_d\in \KK$ such that $\sum_{i=0}^dt\left(c_i\right)\cN^i(\alpha)=0$.
\end{proposition}
\begin{proof}
Write $d'\coloneqq \dim_{\KK}(M)$. Then the set $\left\{m,\partial_M(m),\ldots,\partial^{d'}_M(m)\right\}$ is linearly dependent over $\KK$, in the sense that there exist $c_0,\ldots,c_{d'}$ in $\KK$ such that $mc_0+\partial_M(m)c_1+\cdots+\partial^{d'}_M(m)c_{d'} =0$. If $0 \leq d\leq d'$ is the maximum index for which $c_{d}\neq 0$, we may look at the set $\left\{m,\partial_M(m),\ldots,\partial^{d}_M(m)\right\}$, which is still linearly dependent over $\KK$, and at the relation $mc_0+\partial_M(m)c_1+\cdots+\partial^{d}_M(m)c_{d} =0$ instead. Then
\begin{equation}\label{eq:dern0}
0 = \partial^n_M\left(\sum_{i=0}^{d} \partial^i_M(m)c_i\right) \stackrel{\eqref{eq:deract}}{=} \sum_{i=0}^{d}\sum_{k=0}^{n}\binom{n}{k}\partial^{i+n-k}_M(m)\partial^k(c_i).
\end{equation}
for all $n\geq 0$ and hence
\begin{gather*}
0=\varphi\left(\sum_{i=0}^{d}\sum_{k=0}^{n}\binom{n}{k}\partial^{i+n-k}_M(m)\partial^k(c_i)\right) = \sum_{i=0}^{d}\sum_{k=0}^{n}\binom{n}{k}\varphi\left(\partial^{i+n-k}_M(m)\right)\partial^k(c_i) \\
 = \sum_{i=0}^{d}\sum_{k=0}^{n}\binom{n}{k}\partial^k(c_i)\cN^i\left(\varphi\left(\partial^{\bullet}_M(m)\right)\right)(n-k) = \left(\sum_{i=0}^d t(c_i)\cN^i\left(\varphi\left(\partial^{\bullet}_M(m)\right)\right)\right)(n)
\end{gather*}
as claimed.
\end{proof}

Conversely, pick $\alpha\in\cH(\KK)$ and consider
\begin{equation}\label{eq:alphaf}
\begin{gathered}
\xymatrix@R=0pt@C=10pt{ 
f_{\alpha} \coloneqq \Phi^{-1}(\alpha) : U \ar@{->}[rr] && \KK \\  
\hspace{2.1cm} Y^n \ar@{|->}[rr] && \alpha(n) 
}
\end{gathered}.
\end{equation}
If $\alpha\in \dlin{\KK}$ then, by definition, there exists $\cL=\sum_{i=0}^dt(c_i)\cN^i$ such that $\alpha\in\ker{\cL}$ and hence $f_\alpha$ vanishes on $P_\alpha(Y)U$ where $P_\alpha(Y) = \sum_{i=0}^dY^ic_i$. 
The quotient $U/P_\alpha(Y)U$ is a right $U$-module and we may consider the morphism $\tilde{f}_\alpha \in \Hom{}{}{\KK}{}{U/P_\alpha(Y)U}{\KK}$ induced by $f_\alpha$. We denote by $\bara{1}\coloneqq  1_{\Sscript{U}} + P_\alpha(Y)U$ the equivalence class of the unit $1_{\Sscript{U}}$ of the algebra $U$. 

\begin{proposition}\label{prop:big2}
For $\alpha\in \dlin{\KK}$, the right $\KK$-vector space $M_\alpha\coloneqq U/P_\alpha(Y)U$ is finite-dimensional and a differential $\KK$-vector space with differential $\partial_{M_\alpha}$ given by acting on the right via $Y$. 

If $\alpha$ is of order $d$ (\ie it is annihilated by an operator $\cL$ of order $d$), then $\dim_{\KK}\left(M_\alpha\right)=d$. In particular, the element $\overline{\tilde{f}_\alpha \tensor{T_{M_\alpha}} \cl{1}}$ is a well-defined element in $\Circ{U}$. Moreover, for all $n\in \N$, it satisfies
\begin{equation}\label{eq:ftilde}
\tilde{f}_\alpha\left(\partial_{M_\alpha}^n\left(\cl{1}\right)\right) = \alpha(n).
\end{equation}
\end{proposition}
\begin{proof}
If $P_\alpha(Y)=\sum_{i=0}^dY^ic_i$ and $d=\deg(P_\alpha(Y))$, then we claim that $M\coloneqq M_\alpha=U/P_\alpha(Y)U$ is the right $\KK$-vector space generated by the vectors $y_i\coloneqq Y^i+P_\alpha(Y)U$ for $i=0,\ldots,d-1$. Consider the obvious right $\KK$-linear morphism
\[
\psi:\bigoplus_{i=0}^{d-1}y_i\KK \to \frac{U}{P_\alpha(Y)U}
\]
induced by the inclusions $y_i\KK \subseteq U/P_\alpha(Y)U$. Let us see first that it is injective. Assume that
\[
0=\psi\left(\sum_{j=0}^{d-1}y_jx_j\right) = \sum_{j=0}^{d-1}Y^jx_j + P_\alpha(Y)U
\]
and assume, by contradiction, that $\sum_{j=0}^{d-1}y_jx_j \neq 0$. Then there exists $u\in U$ non-zero such that  $\sum_{j=0}^{d-1}Y^jx_j = \left(\sum_{i=0}^dY^ic_i\right)u$. Let $Y^su_s$ be the leading term of $u$ (the summand with highest degree such that $u_s\neq 0$). Then
\[
\left(\sum_{i=0}^dY^ic_i\right)Y^su_s \stackrel{\eqref{eq:useful}}{=} \sum_{i=0}^d\sum_{k=0}^s\binom{s}{k}Y^{i+k}\partial^{s-k}(c_i)u_s = Y^{d+s}c_du_s + \sum_{i=0}^{d-1}Y^{i+s}c_iu_s + \sum_{i=0}^d\sum_{k=0}^{s-1}\binom{s}{k}Y^{i+k}\partial^{s-k}(c_i)u_s
\]
and hence $Y^{d+s}c_du_s$ is a summand of $\left(\sum_{i=0}^dY^ic_i\right)u$. However, being $s\geq 0$, $Y^{s+d}$ appears with $0$ coefficient in $\sum_{j=0}^{d-1}Y^jx_j$ and hence we should have $c_du_s=0$, which is a contradiction. Summing up, $\sum_{j=0}^{d-1}y_jx_j = 0$ and $\psi$ is injective. 
To see that it is surjective, consider the following facts. First of all, since $c_d\neq 0$, we may assume that $c_d=1$ and hence that $y_d=\sum_{h=0}^{d-1}y_hc_h'$. On the one hand, for every $0\leq j\leq d-1$, $y_j = Y^j + P_\alpha(Y)U = \psi(y_j)$. On the other hand, let us show by induction on $r\geq 0$ that for every $j = d + r$ we have that $y_j\in\sum_{i=0}^{d-i}y_i\KK$. For $r=0$, we know that $y_d=\sum_{h=0}^{d-1}y_hc_h' \in \sum_{i=0}^{d-1}y_i\KK$. Assume then that the property holds for all $0\leq l \leq r-1$ and let us see that it holds for $r$. Compute
\begin{gather*}
y_j = Y^j + P_\alpha(Y)U = Y^dY^{r} + P_\alpha(Y)U = \left(Y^d + P_\alpha(Y)U\right)\triangleleft Y^{r} \\
= \left(\sum_{h=0}^{d-1} Y^hc_h' + P_\alpha(Y)U\right)\triangleleft Y^{r} \stackrel{\eqref{eq:useful}}{=} \sum_{h=0}^{d-1}\sum_{k=0}^{r}\binom{r}{k} Y^{h+k}\partial^{r-k}\left(c_h'\right) + P_\alpha(Y)U \\
= \sum_{h=0}^{d-1}\sum_{k=0}^{r}\binom{r}{k} \left(Y^{h+k} + P_\alpha(Y)U\right)\partial^{r-k}\left(c_h'\right),
\end{gather*}
where $\triangleleft$ is the $U$-action on $M$. Now, for all $h=0,\ldots,d-1$, $k=0,\ldots,r$, we have $h+k\leq d+r-1$ and hence $Y^{h+k} + P_\alpha(Y)U \in \sum_{i=0}^{d-i}y_i\KK$ by the inductive hypothesis. Therefore, 
\[
y_j = \sum_{h=0}^{d-1}\sum_{k=0}^{r}\binom{r}{k} \left(Y^{h+k} + P_\alpha(Y)U\right)\partial^{r-k}\left(c_h'\right) \in \sum_{i=0}^{d-i}y_i\KK
\]
as claimed. By induction, we conclude that $\psi$ is surjective and hence an isomorphism. This allows us to conclude that $M$ is a differential $\KK$-vector space as claimed and the differential is exactly the linear endomorphism induced by right multiplication by $Y$. 

Concerning the last claim, it should be clear now that, for every $n \geq 0$, we have that 
\[
\tilde{f}_\alpha\left(\partial_M^n\left(\cl{1}\right)\right) = \tilde{f}_\alpha\left(y_n\right) = f_{\alpha}(Y^n) = \alpha(n). \qedhere
\]
\end{proof}

As a consequence of Propositions \ref{prop:big1} and \ref{prop:big2} we can prove the main theorem of this section, that allows us to describe the finite dual of a skew polynomial algebra as the space of linear funcionals vanishing on a finite-codimensional ideal and to relate it with differentially recursive sequences.

\begin{theorem}\label{thm:big}
The isomorphism $\Phi$ of Proposition \ref{prop:iso} induces an isomorphism of $\left(\KK\otimes\KK\right)$-algebras $\phi:\Circ{U}\to \dlin{\KK}$ such that $\Phi\circ \zeta = \phi$. It is explicitly given by $\phi\left(\overline{\varphi\tensor{T_M}m}\right) = \left(\varphi\left(\partial_M^\bullet(m)\right)\right)$ with inverse $\phi^{-1}(\alpha) = \overline{\tilde{f}_\alpha\tensor{T_{M_\alpha}}\cl{1}}$, for all $\overline{\varphi\tensor{T_M}m} \in \Circ{U}$ and all $\alpha\in\dlin{\KK}$. It is also an isomorphism as differential $\KK$-algebras.
\end{theorem}

\begin{proof}
It follows from Proposition \ref{prop:big1} that $\Phi\circ \zeta$ lands into $\dlin{\KK}$, thus inducing the claimed morphism $\phi$. Being $\zeta: \Circ{U} \to U^*$ injective, $\phi$ has to be injective. From Proposition \ref{prop:big2} it follows that the assignment $\phi^{-1}:\dlin{\KK}\to\Circ{U},\, \alpha\mapsto \overline{\tilde{f}_\alpha\tensor{T_{M_\alpha} }\cl{1}},$ is a section of $\phi$ (\ie $\phi\circ\phi^{-1} = \id_{\dlin{\KK}}$), whence $\phi$ is surjective as well and hence an isomorphism.
\end{proof}

\begin{corollary}\label{cor:yes}
Up to the canonical morphism $\zeta$, $\Circ{U}$ can be identified with the $\left(\KK\otimes\KK\right)$-subalgebra of $U^*$ of all those linear functionals vanishing on a finite-codimensional (principal) right ideal.
\end{corollary}

\begin{remark}\label{rem:Ulin}
Something more general than \eqref{eq:ftilde} can be said, in light of Theorem \ref{thm:big}. In fact, observe that $\Circ{U}$ is a right $U$-module with action given by
\[
\big(\overline{\varphi \tensor{T_M} m}\big)\triangleleft u = \overline{\varphi \tensor{T_M} \left(m\triangleleft u\right)}
\]
for all $\overline{\varphi \tensor{T_M} m} \in \Circ{U}$, $u\in U$, and, as such, it is an $U$-submodule of $U^*$. Therefore, $\dlin{\KK}$ inherits a structure of right $U$-module such that $\phi$ is $U$-linear and, in particular, it becomes a right $U$-submodule of $\cH(\KK)$. Therefore, for every $i\geq 0$ we have
\begin{equation}\label{eq:suspect}
\left(\tilde{f}_\alpha\left(\partial_{M_\alpha}^\bullet\left(y_i\right)\right)\right) = \phi\left(\overline{\tilde{f}_\alpha\tensor{T_{M_\alpha}}y_i}\right) = \phi\left(\overline{\tilde{f}_\alpha\tensor{T_{M_\alpha}}1}\triangleleft Y^i\right) = \alpha\triangleleft Y^i = \cN^i(\alpha).
\end{equation}
\end{remark}


As a consequence of Theorem \ref{thm:big} we can finally state the central result of the paper.

\begin{theorem}\label{thm:chalgd}
The $\K$-algebra $\dlin{\KK}$ of differentially recursive sequences enjoys a structure of commutative Hopf algebroid over $\KK$. The structure maps are explicitly given by the source $s$, the target $t$, the counit $\varepsilon:\dlin{\KK} \to \KK, \alpha\mapsto \alpha(0)$, the comultiplication
\begin{equation}\label{eq:comult}
\Delta:\dlin{\KK} \longrightarrow  \dlin{\KK}\tensor{\KK}\dlin{\KK}, \qquad \alpha\longmapsto \sum_{i=0}^{d-1} \cN^i\left(\alpha\right) \tensor{\KK} \left(y_i^*\left(y_\bullet\right)\right),
\end{equation}
where $\left\{y_i = Y^i + P_\alpha(Y)U,y_i^* \, \right\}_{i=0,\, \ldots,\, d-1}$ is the dual basis of $M_\alpha$ and $d$ is the degree of $\alpha$. Lastly, the antipode is given by 
\[
\sS:\dlin{\KK}\longrightarrow \dlin{\KK}, \qquad \alpha \longmapsto \left(\sum_{k=0}^n\binom{n}{k}(-1)^{n-k}\partial^k\left(\alpha(n-k)\right)\right)_{n\geq 0}.
\]
\end{theorem}

\begin{proof}
The structures come from those of $\Circ{U}$ via $\phi$ of Theorem \ref{thm:big}.
\end{proof}

\begin{remark}
We know, for abstract reasons, that if $\alpha,\beta\in\dlin{\KK}$, then $\alpha\cdot\beta\in\dlin{\KK}$ as well. However, before proceeding, the reader may be interested in knowing explicitly which differential operator the product of two differentially recursive sequences satisfies\footnote{Apart from its own interest, this computation could perhaps have a certain interest form a combinatorial point of view.}.

To this aim, assume that $\cL_\alpha(\alpha) = 0 = \cL_\beta(\beta)$ with $\cL_\alpha \coloneqq \sum_{i=0}^{d_\alpha} t(a_i)\cN^i$ and $\cL_\beta \coloneqq \sum_{i=0}^{d_\beta} t(b_i)\cN^i$, $a_i,b_j\in\KK$. In light of Theorem \ref{thm:big},
\[
\alpha \cdot \beta = \phi\left( \phi^{-1}(\alpha) \cdot \phi^{-1}(\beta) \right) = \phi\left( \left(\overline{\tilde{f}_{\alpha} \tensor{T_{M_{\alpha}}} \cl{1}}\right) \cdot \left(\overline{\tilde{f}_{\beta} \tensor{T_{M_{\beta}}} \cl{1}}\right) \right) \stackrel{\eqref{eq:multcirc}}{=} \phi\left(\overline{ \left(\tilde{f}_\beta \diamond \tilde{f}_\alpha\right) \tensor{T_{M_\beta\tensor{\KK}M_\alpha}} \left(\cl{1}\tensor{\KK}\cl{1}\right) }\right).
\]
By Proposition \ref{prop:big1}, $\alpha\cdot \beta$ is annihilated by an operator $\cL_{\alpha\cdot\beta}$ of order $d$ such that
\[
0\leq d\leq \dim_{\KK}\left(M_\beta\tensor{\KK}M_\alpha\right) = d_\alpha d_\beta
\]
(see Proposition \ref{prop:big2} as well). Therefore, we may assume $\cL_{\alpha\cdot \beta} = \sum_{i=0}^{d_\alpha d_\beta}t(x_i)\cN^i$ for some (not necessarily non-zero) $x_i\in \KK$, $i=0,\ldots,d_\alpha d_\beta$. Consider the (non-linear) system of equations
\begin{equation}\label{eq:supersyst}
 0 = \cL_{\alpha\cdot \beta}\left(\alpha\cdot \beta\right)(n) = \sum_{i=0}^{d_{\alpha}d_{\beta}}\sum_{k=0}^{n}\binom{n}{k}\partial^k\left(x_i\right)\left(\alpha\cdot\beta\right)(n-k+i), \qquad 0\leq n\leq d_\alpha d_\beta-1,
\end{equation}
in the $d_{\alpha}d_{\beta}+1$ unknowns $x_0,\ldots,x_{d_{\alpha}d_{\beta}}$. By Proposition \ref{propApp:supersysts}, this is equivalent to the homogeneous (linear) system
\begin{equation}\label{eq:supersyst2}
\sum_{i=0}^{d_{\alpha}d_{\beta}}x_i\left(\sum_{k=0}^n\binom{n}{k}(-1)^k\partial^{k}\left((\alpha\cdot \beta)(n-k+i)\right)\right) = 0, \qquad 0\leq n\leq d_{\alpha}d_{\beta}-1,
\end{equation}
which necessarily admits a non-zero solution (recall that $\alpha\cdot\beta\in\dlin{\KK}$, whence $\cL_{\alpha\cdot\beta}$ exists and it satisfies \eqref{eq:supersyst}). Conversely, in light of Lemma \ref{lemmaApp:vanish}, any solution of \eqref{eq:supersyst2} (and hence of \eqref{eq:supersyst}) gives rise to an operator $\cL'$ which annihilates $\alpha\cdot \beta$. Therefore, in order to find an operator $\cL_{\alpha\cdot \beta}$ such that $0 = \cL_{\alpha\cdot \beta}(\alpha\cdot \beta)$ it is enough to solve \eqref{eq:supersyst2}.
\end{remark}

Let us show the procedure in detail on a concrete and easy handled example.

\begin{example}\label{eq:super1}
Assume that $0=\alpha\triangleleft (Y-p)$, $0=\beta\triangleleft(Y-q)$, $\alpha(0) \neq 0$ and $\beta(0) \neq 0$. Then  the recursions are 
\[
\alpha(n+1) = \sum_{k=0}^n\binom{n}{k}\partial^k\left(p\right)\alpha(n-k) \qquad \text{and} \qquad \beta(n+1) = \sum_{k=0}^n\binom{n}{k}\partial^k\left(q\right)\beta(n-k)
\]
for all $n\geq 0$. In particular, one may check directly that
\begin{align*}
\alpha(1) & = p\alpha(0) & \beta(1) & = q\beta(0) \\
\alpha(2) & = \left(p^2+\partial(p)\right)\alpha(0) & \beta(2) & = \left(q^2+\partial(q)\right)\beta(0)
\end{align*}
by iterative substitution (we will come back on these computations with more detail in \S\ref{sec:matrixdiffeq}). Now, consider the relation $\left(\alpha\cdot \beta\right)\triangleleft\left(Ya+b\right)=0$. For $n=0$, this gives rise to the equation
\[
a\left(\alpha(1)\beta(0)+\alpha(0)\beta(1)\right) + b\alpha(0)\beta(0) = 0
\]
which, after substituting and cancelling $\alpha(0)\beta(0)$, becomes
\[
a\left(p+q\right) + b = 0.
\]
For $n=1$ it gives rise to
\[
\begin{split}
a\left(\alpha(2)\beta(0)+2\alpha(1)\beta(1)+\alpha(0)\beta(2)\right) + \left(\partial(a)+b\right)\left(\alpha(1)\beta(0)+\alpha(0)\beta(1)\right) + \partial(b)\alpha(0)\beta(0) = 0,
\end{split}
\]
which, after substituting and cancelling $\alpha(0)\beta(0)$, becomes
\[
\begin{split}
a\left(p^2+\partial(p)+2pq+q^2+\partial(q)\right) + \left(\partial(a)+b\right)\left(p+q\right) + \partial(b) = 0.
\end{split}
\]
Observe that, since $b= - a\left(p+q\right)$, we get
\begin{gather*}
a\left(p^2+\partial(p)+2pq+q^2+\partial(q)\right) + \partial(a)\left(p+q\right) + \\
-a\left(p^2+2pq+q^2\right) - \partial(a)\left(p+q\right) - a(\partial(p)+\partial(q)) = 0,
\end{gather*}
that is to say, the second equation is identically satisfied and $Ya+b$ is determined by any solution of
\[
a\left(p+q\right) + b = 0,
\]
as we were expecting. The easiest one is, of course, $a=1$ and $b = -(p+q)$, \ie $P_{\alpha\cdot\beta}(Y) = Y -(p+q)$. 
\end{example}

Let us devote an additional bit of time to see a second, more meaningful, example with a bit less of details.

\begin{example}\label{eq:super2}
Assume that $0=\alpha\triangleleft (Y-p)$, $0=\beta\triangleleft(Y^2-Yq_1-q_0)$, $\alpha(0) = a \neq 0$ and $(\beta(0),\beta(1)) = (b_0,b_1) \neq (0,0)$. Assume also that $a=1$. As in Example \ref{eq:super1}, this is not restrictive. Then the recursions are 
\[
\alpha(n+1) = \sum_{k=0}^n\binom{n}{k}\partial^k\left(p\right)\alpha(n-k) \qquad \text{and} \qquad \beta(n+2) = \sum_{k=0}^n\binom{n}{k}\partial^k\left(q_1\right)\beta(n-k+1) + \sum_{k=0}^n\binom{n}{k}\partial^k\left(q_0\right)\beta(n-k)
\]
for all $n\geq 0$. In particular, one may check directly that
\begin{align*}
\alpha(0) & = 1, & \beta(0) & = b_0 \\
\alpha(1) & = p, & \beta(1) & = b_1 \\
\alpha(2) & = p^2+\partial(p), & \beta(2) & = q_1b_1 + q_0b_0 \\
\alpha(3) & = p^3+3p\partial(p)+\partial^2(p), & \beta(3) & = \left(q_1^2+\partial(q_1)+q_0\right)b_0 + \left(q_1q_0 + \partial(q_0)\right)b_1 \\
\alpha(4) & = p^4 + 6 p^2 \partial(p) + 4p\partial^2(p) + 3(\partial(p))^2 + \partial^3(p), & \beta(4) & = \left(q_1^3+3q_1\partial(q_1) + 2q_0q_1 + \partial^2(q_1) + 2\partial(q_0)\right)b_1 + \\ 
 & & & \phantom{=} + \left(q_1^2q_0 + q_1\partial(q_0) + 2q_0\partial(q_1) + q_0^2 +\partial^2(q_0)\right)b_0
\end{align*}
by iterative substitution as before. Now, consider the relation $\left(\alpha\cdot \beta\right)\triangleleft\left(Y^2a+Yb+c\right)=0$. For $n=0$, this gives rise to the equation
\begin{equation}\label{eq:n0}
\left((p^2+\partial(p)+q_0)b_0 + (2p+q_1)b_1\right)a + \left(pb_0+b_1\right)b + b_0c = 0.
\end{equation}
For $n=1$ it gives rise to
\begin{equation}\label{eq:n1}
\begin{gathered}
\left((p^3+3p\partial(p)+\partial^2(p) +3pq_0 +q_1q_0 + \partial(q_0))b_0 + (3p^2+3\partial(p) + 3pq_1 + q_1^2+\partial(q_1) + q_0) b_1\right)a + \\ 
+ \left((p^2+\partial(p)+q_0)b_0+(2p+q_1)b_1\right)\left(\partial(a) + b\right) + \left(pb_0+b_1\right)(c+\partial(b)) + b_0\partial(c) = 0.
\end{gathered}
\end{equation}
For the sake of brevity, we omit here the expression for $n=2$. Now, a priori this situation gives rise to three different possibilities.
\begin{enumerate}
\item $b_0=0$ and we may consequently assume $b_1=1$. In this case, \eqref{eq:n0} gives $b = -(2p+q_1)a$. If we substitute it in \eqref{eq:n1} then we get $c = (p^2+pq_1-q_0-\partial(p))a$. A straightforward but tedious computation confirms that the third equation is identically satisfied and hence
\begin{equation}\label{eq:1sol}
P_{\alpha\cdot\beta}(Y) = Y^2 + Y(2p+q_1) - \left(p^2+pq_1-q_0-\partial(p)\right).
\end{equation}
\end{enumerate}
Having ruled out the case $b_0=0$, we may assume henceforth that $b_0=1$ and substitute $b_1$ with $w:=b_1/b_0$. Under these hypothesis, \eqref{eq:n0} gives
\begin{equation}\label{eq:zRiccati}
c = -\left(p^2+\partial(p)+q_0+(2p+q_1)w\right)a -(p+w)b.
\end{equation}
\begin{enumerate}[resume*]
\item $w$ satisfies the (generalized) Riccati equation $\partial(w) = q_0+wq_1-w^2$ (see \cite[\S I.1]{Reid}). In this case, \eqref{eq:n1} turns out to be automatically satisfied and hence $P_{\alpha\cdot\beta}$ is determined simply by \eqref{eq:zRiccati}. Therefore
\[
P_{\alpha\cdot\beta} (Y) = Y - (p+w).
\]
This is a case in which the order of the product is strictly smaller than the product of the orders.
\item $w$ does not satisfy the (generalized) Riccati equation, \ie $\partial(w) \neq q_0+wq_1-w^2$. In this case, by substituting \eqref{eq:zRiccati} into \eqref{eq:n1} we find out that $b = - (2p+q_1)a$ and that the third equation is identically satisfied (as expected). Therefore,
\[
P_{\alpha\cdot\beta}(Y) = Y^2 + Y(2p+q_1) - \left(p^2 + pq_1-q_0-\partial(p)\right),
\]
which coincides with \eqref{eq:1sol}.
\end{enumerate}
\end{example}

Back to the main topic, Corollary \ref{cor:yes} gives a description of $\Circ{U}$ that closely resembles the classical one: for an ordinary Hopf $\K$-algebra $H$, $\Circ{H}$ is the subalgebra of $H^*$ of all those linear functionals that vanishes on a finite-codimensional (two-sided) ideal. As a consequence, a very natural question arises. Recall from \cite[\S4]{ArdiLaiachiPaolo} that there exists a second finite dual construction for cocommutative Hopf algebroids, obtained via the Special Adjoint Functor Theorem. This alternative finite dual $\saft{U}$ is, in a suitable sense, the biggest $\KK$-coring inside $U^*$. Namely, it is uniquely determined by the following universal property: $\saft{U}$ is a $\KK$-coring together with a $\K$-linear map $\xi:\saft{U}\to U^*$ which satisfies
\begin{equation}\label{eq:saft}
\xi(z)(uv) = \sum \xi(z_{(1)})\left(\tau\left(\xi\left(z_{(2)}\right)(u)\right)v\right), \quad \xi(z)(1_U) = \varepsilon(z)\quad \text{and} \quad \xi(x\cdot z \cdot y)(u) = x\xi(z)(\tau(y)u)
\end{equation}
for all $z\in \saft{U}$, $u,v\in U$ and $x,y\in\KK$. It is universal with respect to this property, in the sense that if $C$ is another $\KK$-coring with a $\K$-linear morphism $f:C\to U^*$ satisfying \eqref{eq:saft} then there exists a unique coring homomorphism $\what{f} : C\to \saft{U}$ such that $\xi\circ\what{f} = f$. In light of this, it is natural to ask how $\Circ{U}$ (and $\dlin{\KK}$) are related with $\saft{U}$, which will give the second universal property mentioned in the Introduction.

\begin{theorem}\label{thm:bullet}
For a differential field $(\KK,\partial)$, the two-sided $\KK$-vector space of differentially recursive sequences $\dlin{\KK}$ with the inclusion $\xi:\dlin{\KK}\subseteq U^*$, satisfies the universal property of the Hopf algebroid $\saft{U}$. In particular, for $U=\KK[Y;\partial]$,  we have a chain of isomorphisms $\Circ{U}\cong \dlin{\KK}\cong \saft{U}$ of Hopf algebroids.
\end{theorem}

\begin{proof}
Assume that $C$ is a $\KK$-coring together with a $\K$-linear morphism $g:C\to U^*$ satisfying \eqref{eq:saft}. For every $c\in C$, write explicitly $\Delta(c) = \sum_{j=1}^r c_j' \tensor{\KK} c_j''$. All the morphisms $g\left(c_j''\right)$ for $j=1,\ldots,r$ admit a kernel $\ker{g\left(c_j''\right)}\subseteq U$ which is of codimension 1. In particular, since there is only a finite number of them and $U$ is infinite-dimensional, $\bigcap_{j=1}^r\ker{g\left(c_j''\right)} \ni P_c(Y)\neq 0$. In light of \eqref{eq:saft}, for every $n\geq 0$
\[
g(c)\left(P_c(Y)Y^n\right) = \sum_{j=1}^r g\left(c_j'\right)\left(\tau\left(g\left(c_j''\right)\left(P_c(Y)\right)\right)Y^n\right) = 0
\]
and hence $\ker{g(c)}\supseteq P_c(Y)U$, which is a finite-codimensional (principal) right ideal. Summing up, $g$ factors uniquely through $g':C\to\dlin{\KK}, c\mapsto g'_c$, where $\xi(g'_c)=g(c)$ for all $c\in C$ and we have the following chain of equalities
\begin{equation}\label{eq:clarify}
g(c)(Y^n) = \xi(g'_c)(Y^n) = g'_c(n) = f_{g'_c}(Y^n) = \tilde{f}_{g'_c}(y_n), \qquad \forall\,n\in \N,
\end{equation} 
where $\tilde{f}_{g'_c}\in \left(U/P_{g'_c}(Y)U\right)^*$ is the factorization through the quotient of $f_{g'_c} = \xi(g'_c)$.
Let us check that with the structure introduced in Theorem \ref{thm:chalgd} and with the canonical inclusion $\xi:\dlin{\KK}\to U^*,\alpha\mapsto f_\alpha$, as $\KK$-coring $\dlin{\KK}$ satisfies relations \eqref{eq:saft}. First of all, $\xi$ is $\KK$-bilinear because $\Phi$ of Proposition \ref{prop:iso} is.
Secondly, in light of \eqref{eq:alphaf}, $\xi(\alpha)(1_u) = \alpha\left(0\right) = \varepsilon(\alpha)$. Thirdly, we have that
\begin{gather*}
\sum_{(\alpha)}\xi\left(\alpha_{(1)}\right)\left(\xi\left(\alpha_{(2)}\right)\left(Y^k\right)Y^n\right) \stackrel{\eqref{eq:comult}}{=} \sum_{i=0}^{d-1} \xi\left(\cN^i\left(\alpha\right)\right)\left(\xi\Big(\left(y_i^*\left(y_\bullet\right)\right)\Big)\left(Y^k\right)Y^n\right) \\
 \stackrel{\eqref{eq:alphaf}}{=} \sum_{i=0}^{d-1} \xi\left(\cN^i\left(\alpha\right)\right)\left(y_i^*\left(y_k\right)Y^n\right) \stackrel{\eqref{eq:useful}}{=} \sum_{i=0}^{d-1} \xi\left(\cN^i\left(\alpha\right)\right)\left(\sum_{j=0}^n\binom{n}{j}Y^j\partial^{n-j}\left(y_i^*\left(y_k\right)\right)\right) \\
 = \sum_{i=0}^{d-1}\sum_{j=0}^n\binom{n}{j} \xi\left(\cN^i\left(\alpha\right)\right)\left(Y^j\right)\partial^{n-j}\left(y_i^*\left(y_k\right)\right) \stackrel{\eqref{eq:alphaf}}{=} \sum_{i=0}^{d-1}\sum_{j=0}^n\binom{n}{j} \cN^i(\alpha)(j)\partial^{n-j}\left(y_i^*\left(y_k\right)\right) \\
 \stackrel{\eqref{eq:suspect}}{=} \sum_{i=0}^{d-1}\sum_{j=0}^n\binom{n}{j} \tilde{f}_\alpha\left(\partial_{M_\alpha}^j\left(y_i\right)\right)\partial^{n-j}\left(y_i^*\left(y_k\right)\right) \stackrel{(*)}{=} \tilde{f}_\alpha\left(\sum_{i=0}^{d-1}\sum_{j=0}^n\binom{n}{j} \partial_{M_\alpha}^j\left(y_i\right)\partial^{n-j}\left(y_i^*\left(y_k\right)\right)\right) \\
 \stackrel{\eqref{eq:deract}}{=} \tilde{f}_\alpha\left(\partial_{M_\alpha}^n\left(\sum_{i=0}^{d-1} y_iy_i^*\left(y_k\right)\right)\right) = \tilde{f}_\alpha\left(\partial_{M_\alpha}^{n}\left(y_k\right)\right) \stackrel{\eqref{eq:suspect}}{=} \alpha(n+k) = \xi(\alpha) \left(Y^{n+k}\right)
\end{gather*}
where in $(*)$ we used the right $\KK$-linearity of $\tilde{f}_\alpha$. Therefore, $\xi$ satisfies the condition \eqref{eq:saft}. We are left to check that the morphism $g':C\to \dlin{\KK}, c\mapsto g'_c,$ where $g'_c(n)=g(c)(Y^n)$ for all $n\geq 0$, is a morphism of $\KK$-corings. Again, it is $\KK$-bilinear because $g:C\to U^*$ was and $\Phi$ is. It is counital because
\[
\varepsilon\left(g'_c\right) = g'_c(0) = g(c)(1_U) \stackrel{\eqref{eq:saft}}{=} \varepsilon_C(c).
\]
It is comultiplicative because
\begin{align*}
& \Delta\left(g'_c\right) \stackrel{\eqref{eq:comult}}{=} \sum_{i=0}^{d_c-1} \cN^i(g'_c) \tensor{\KK} \left(y_i^*\left(y_\bullet\right)\right) \stackrel{\eqref{eq:clarify}}{=} \sum_{i=0}^{d_c-1} \left(g(c)(Y^iY^\bullet)\right) \tensor{\KK} \left(y_i^*\left(y_\bullet\right)\right) \\
& \stackrel{\eqref{eq:saft}}{=} \sum_{i=0}^{d_c-1} \left(g(c_{(1)})\left(g(c_{(2)})\left(Y^i\right)Y^\bullet\right)\right) \tensor{\KK} \left(y_i^*\left(y_\bullet\right)\right) \stackrel{\eqref{eq:useful}}{=} \sum_{i=0}^{d_c-1} \left(g(c_{(1)})\left(\sum_{k=0}^{n}\binom{n}{k}Y^k\partial^{n-k}\left(g'_{c_{(2)}}\left(i\right)\right)\right)\right)_{n\geq0} \tensor{\KK} \left(y_i^*\left(y_\bullet\right)\right) \\
& \stackrel{\eqref{eq:clarify}}{=} \sum_{i=0}^{d_c-1} \left(\sum_{k=0}^{n}\binom{n}{k}g'_{c_{(1)}}\left(k\right)\partial^{n-k}\left(g_{c_{(2)}}\left(i\right)\right)\right)_{n\geq 0} \tensor{\KK} \left(y_i^*\left(y_\bullet\right)\right) \stackrel{\eqref{eq:Hurwitz}}{=} \sum_{i=0}^{d_c-1} g'_{c_{(1)}}\cdot t\left(g'_{c_{(2)}}\left(i\right)\right) \tensor{\KK} \left(y_i^*\left(y_\bullet\right)\right) \\
& = \sum_{i=0}^{d_c-1} g'_{c_{(1)}} \tensor{\KK} s\left(g'_{c_{(2)}}\left(i\right)\right)\cdot \left(y_i^*\left(y_\bullet\right)\right) \stackrel{\eqref{eq:clarify}}{=} \sum_{i=0}^{d_c-1} g'_{c_{(1)}} \tensor{\KK} \left(\tilde{f}_{g'_{c_{(2)}}}\left(y_i\right)y_i^*\left(y_\bullet\right)\right) = \sum_{(c)} g'_{c_{(1)}} \tensor{\KK} \left(\tilde{f}_{g'_{c_{(2)}}}\left(y_\bullet\right)\right) \\
& \stackrel{\eqref{eq:clarify}}{=} \sum_{(c)} g'_{c_{(1)}} \tensor{\KK} \left(g\left(c_{(2)}\right)\left(Y^\bullet\right)\right) = \sum_{(c)} g_{c_{(1)}} \tensor{\KK} g_{c_{(2)}} = (g\tensor{\KK}g)\left(\Delta_C(c)\right),
\end{align*}
where $d_c$ is the degree of $g'_c$. Therefore, it is a coring homomorphism and the proof is complete.
\end{proof}

%

\subsection{Comparing linearly and differentially recursive sequences}

Let us conclude the section by showing how the notion of differentially recursive sequences is related to the classical notion of linearly recursive sequences over fields. Given any field $\FF$, recall that a sequence $\alpha\in\FF^\N$ is linearly recursive if there exist $d\geq 0$ and coefficients $b_0,\ldots,b_{d-1}\in\FF$ such that
\[
\alpha(n+d) = b_{d-1}\alpha(n+d-1) + \cdots + b_0\alpha(n)
\]
for all $n\geq 0$. The space of linearly recursive sequences over $\FF$ will be denoted by $\Lin{\FF}$. Notice that $\alpha$ is linearly recursive  if and only if it satisfies $\cL(\alpha)=0$ where $\cL=\sum_{i=0}^{d}s(c_i)\cN^i$ and $c_i\in\FF$ for all $i=0,\ldots,d$. If $\FF$ is considered as differential field with the zero derivation, then the morphisms $s,t:\FF \to \cH(\FF)$ obviously coincide. Therefore, a linearly recursive sequence over $\FF$ (in the classical sense) is the same as a differentially recursive sequence over $(\FF,0)$ (in the sense of the present paper). Thus, the notion of differentially recursive sequence does not add anything new to the picture when the base field is differential with zero derivation.  

In general, for a differential field $(\KK,\partial)$, the spaces $\dlin{\KK}$ and $\Lin{\KK}$ are related  by the following commutative diagram of $\K$-algebras
\begin{equation}\label{eq:diag}
\begin{gathered}
\xymatrix @=10pt{
 & \cH(\KK) & \\
\dlin{\KK} \ar[ur] & & \Lin{\KK} \ar[ul] \\
 & \Lin{\K} \ar[ul] \ar[ur] & 
}
\end{gathered}
\end{equation}
where, as usual, $\K$ is the field of constants of $\left(\KK,\partial\right)$ and all the morphisms are injective. In the forthcoming Example \ref{exam:Dlink} we show that, in general, $\dlin{\KK}$ and $\Lin{\KK}$ have different images inside $\cH(\KK)$.

\begin{remark}\label{rem:polykilllin}
Consider the assignment $\lambda_s : \KK \to \End{\K}{\cH(\KK)}$ sending every $x\in\KK$ to the endomorphism $\lambda_{s(x)}:\cH(\KK)\to\cH(\KK), \alpha\mapsto s(x)\alpha =(x\alpha_n)_{\Sscript{n \, \in \, \mathbb{N}}}$. It is a morphism of $\K$-algebras satisfying $\cN\circ \lambda_{s(x)} = \lambda_{s(x)} \circ \cN$ for every $x\in\KK$. Therefore, it extends to a unique $\K$-algebra morphism $\Lambda: \KK[Z] \to \End{\K}{\cH(\KK)}$ which makes of $\cH(\KK)$ a left $\KK[Z]$-module with action
\[
P(Z)\triangleright \alpha = \cL(\alpha),
\]
where $\cL = \sum_{j=0}^ds(p_j)\cN^j$ if $P(Z) = \sum_{j=0}^d p_jZ^j$. With this interpretation, $\alpha\in\Lin{\KK}$ if and only if $P(Z)\triangleright \alpha = 0$ for some $P(Z)\in\KK[Z]$.
However, this procedure does not convert $\cH(\KK)$ into a $\left(\KK[Z],\KK\right)$-bimodule, as
\[
Z\triangleright\left(\alpha\triangleleft x\right) = Z \triangleright \left(t(x)\cdot\alpha\right) = \cN\left(t(x)\cdot\alpha\right) = t\left(\partial(x)\right)\cdot \alpha + t(x)\cdot \cN\left(\alpha\right)
\]
while
\[
\left(Z\triangleright\alpha\right)\triangleleft x = t(x)\cdot \cN\left(\alpha\right)
\]
for all $x\in\KK$, $\alpha\in\cH(\KK)$. In particular, $\cH(\KK)$ is not a $\left(\KK[Z],\KK[Y;\partial]\right)$-bimodule.
\end{remark}

\begin{example}\label{exam:Dlink}
In general, there is no evident relation between $\dlin{\KK}$ and $\Lin{\KK}$. 
Consider, for example, the case $\KK=\C(z)$ with the usual derivative $\partial_z$. If we pick $1/z\in\C(z)$, then
\[
t\left(\frac{1}{z}\right) = \left((-1)^n\frac{n!}{z^{n+1}}\right) = \left(\frac{1}{z},-\frac{1}{z^2},\frac{2}{z^3},-\frac{6}{z^4},\cdots\right) \in\cH\left(\C(z)\right)
\]
and satisfies $t(1/z)\triangleleft \big(\partial_z(1/z) - Y(1/z)\big) = 0$, whence it belongs to $\dlin{\C(z)}$. However, let us show that it cannot belong to $\Lin{\C(z)}$ by mimicking the proof of \cite[Lemma B.6]{ArdiLaiachiPaolo-DQ}. Observe that if we assume that 
\begin{equation}\label{eq:gatoencerrado}
0 = \sum_{i=0}^d s(c_i)\cN^i\left(t\left(\frac{1}{z}\right)\right) = \left(\sum_{i=0}^d(-1)^{n+i}\frac{(n+i)!c_i}{z^{n+i+1}}\right)_{n\geq0} = \left(\frac{\sum_{i=0}^d(-1)^{n+i}(n+i)!c_iz^{d-i}}{z^{n+d+1}}\right)_{n\geq 0},
\end{equation}
then, in particular, the set of elements $\left\{c_0z^d,\ldots,c_{d-1}z,c_d\right\}\subseteq \C(z)$ satisfy the $(d+1)\times (d+1)$ linear system:
\[
\begin{pmatrix} 
0! & -1! & 2! & \cdots & (-1)^dd! \\ 
-1! & 2! & -3! & \cdots & (-1)^{d+1}(d+1)! \\ 
2! & -3! & 4! & \cdots & (-1)^{d+2}(d+2)! \\
\vdots & \vdots & \vdots & \ddots & \vdots \\
(-1)^dd! & (-1)^{d+1}(d+1)! & (-1)^{d+2}(d+2)! & \cdots & (-1)^{2d}(2d)!
\end{pmatrix}
\cdot \left(\begin{array}{c} c_0z^d \\ c_1z^{d-1} \\ \vdots \\ c_{d-1}z \\ c_d\end{array}\right) = 0
\]
The matrix of this system is $T=\Big((-1)^{i+j}(i+j)!\Big)_{i,j}$, for $i,j$ that run from $0$ to $d$ and it satisfies
\begin{align*}
\det(T) & = \det\Big((-1)^{i+j}(i+j)!\Big) = \det\Big((-1)^{i+j}i!j!q_{ij}\Big) = \left(\prod_{i=0}^d(-1)^ii!\right)\left(\prod_{j=0}^d(-1)^jj!\right)\det(Q_d) \\
 & = \left((-1)^{\frac{d(d+1)}{2}}\prod_{i=0}^di!\right)^2\det(Q_d) = \left(0!1!\cdots d!\right)^2\det(Q_d)
\end{align*}
where $q_{ij}=\binom{i+j}{i}$ and $Q_d$ is the $d$-th Pascal matrix. In view of \cite[Discussion preceding Theorem 4]{BrawerPivorino}, we know that $\det(Q_d)=1$, whence $\det(T) \neq 0$ and hence $T$ is invertible. As a consequence, the only solution turns out to be $c_i=0$ for all $i=0,\ldots,d$, 
and so there is no non-trivial relation of the form \eqref{eq:gatoencerrado}. 

In the other way around,  consider the sequence $\alpha = \left(z^\bullet\right)=(1,z,z^2,\ldots)\in\C(z)^\N$. This is linearly recursive since it satisfies $(z-Z)\triangleright \alpha = 0$, but it cannot be differentially recursive because of the following argument. Assume, by contradiction, that there exists $P(Y) = \sum_{i=0}^dY^ic_i \in \C(z)[Y;\partial]$ with $c_d\neq 0$ such that $\alpha\triangleleft P(Y) = 0$. This implies that
\[
0=\left(\sum_{i=0}^d\sum_{k=0}^n\binom{n}{k}\partial_z^k(c_i)z^{n-k+i}\right)_{n\geq 0} = \left(\sum_{k=0}^n\binom{n}{k}\left(\sum_{i=0}^d\partial_z^k(c_i)z^i\right)z^{n-k}\right)_{n\geq 0}.
\]
By induction on $m$, one deduces from this that $\sum_{i=0}^d\partial_z^m(c_i)z^i=0$ for all $m\geq 0$.
Now, from $\sum_{i=0}^d\partial_z^m\left(c_i\right)z^i = 0$ we deduce that
\[
0 = \partial_z\left(\sum_{i=0}^d\partial_z^m\left(c_i\right)z^i\right) = \sum_{i=0}^d\partial_z^{m+1}\left(c_i\right)z^i + \sum_{i=1}^di\partial_z^m\left(c_i\right)z^{i-1} = \sum_{i=1}^di\partial_z^m\left(c_i\right)z^{i-1}
\]
and hence $\sum_{i=1}^di\partial_z^m\left(c_i\right)z^{i-1} = 0$ for all $m\geq 0$. In the same way,
\[
0 = \partial_z\left(\sum_{i=1}^di\partial_z^m\left(c_i\right)z^{i-1}\right) = \sum_{i=1}^di\partial_z^{m+1}\left(c_i\right)z^{i-1} + \sum_{i=2}^di(i-1)\partial_z^m\left(c_i\right)z^{i-2},
\]
so that $\sum_{i=2}^di(i-1)\partial_z^m\left(c_i\right)z^{i-2} = 0$ for every $m\geq 0$. By proceeding inductively, after $d$ steps one concludes that $d!\partial_z^m(c_d) = 0$ for every $m\geq 0$. In particular, for $m=0$ we find $d!c_d = 0$, which contradicts the choice of $c_d$. Thus, there does not exist $P(Y)\neq 0$ such that $\alpha\triangleleft P(Y) = 0$. 

Summing up, there exist linearly recursive sequences that are not differentially recursive and conversely. In addition, there exist sequences which are linearly and differentially recursive but that are not coming from $\Lin{\C}$. The easiest example is $s(z) = (z,0,\ldots)$: it satisfies $\cN\left(s(z)\right) = 0$, which identifies it as a linearly and a differentially recursive sequence ($Z\triangleright s(z) = 0 = s(z) \triangleleft Y$), but it is not an element of $\C^\N$. In fact, any sequence with compact (finite) support would be linearly and differentially recursive without being in $\C^\N$.
\end{example}


\section{Connections with linear differential matrix equations}\label{sec:matrixdiffeq}

This  section is devoted to explain  how differential linear matrix equations can be approached by means of differentially recursive sequences.  Firstly we show that the space of these sequences can be seen as a direct limit of all spaces of formal solutions of linear homogeneous differential equations (that is, it is a kind of  ``universal formal solution''). Secondly we comment on how Picard-Vessiot ring extensions can be constructed from the Hopf algebroid of all differentially recursive sequences by analysing, for the sake of simplicity, the case of two-dimensional differential vector spaces. 

As before we fix a differential field $(\KK, \partial)$ with (non-trivial) sub-field of constants $\K$ and we consider its differential algebra of Hurwitz series $(\cH(\KK), \cN)$ and its algebra of differential operators $U:=\KK[Y, \partial]$.


\subsection{$\dlin{\KK}$ as the universal algebroid of solutions}

Consider a differential operator $\cL = \sum_{i=0}^d t(c_i)\cN^i$ and $P_{\cL}(Y) \coloneqq  \sum_{i=0}^d Y^ic_i$ the associated element in $U$. For any $\alpha\, \in \, \cH(\KK)$ solution of $\cL(\alpha)=0$, we have that the map $f_{\alpha}$ of equation \eqref{eq:alphaf} vanishes on the right ideal  $P_{\cL}(Y)U$ and conversely. Therefore, the following correspondences are bijective
\begin{equation}\label{eq:LaBrocante}
\begin{gathered}
\xymatrix@R=0pt{
\ker{\cL} \ar@{->}^-{\cong}[r] & \left(P_{\cL}(Y)U\right)^\perp \ar@{->}^-{\cong}[r] & \left(\frac{U}{P_{\cL}(Y)U}\right)^*=\hom{\KK}{ \frac{U}{P_{\cL}(Y)U}}{\KK}\\
\alpha \ar@{|->}[r] & f_\alpha \ar@{|->}[r] & \tilde{f}_\alpha
},
\end{gathered}
\end{equation}
where $\left(P_{\cL}(Y)U\right)^\perp = \left\{f:U\to \KK\mid f\left(P_{\cL}(Y)U\right) = 0\right\}$. In light of this, we refer to $M_\cL^*\coloneqq \left({U}/{P_{\cL}(Y)U}\right)^*$ as the space of solutions of the differential equation $\cL(\alpha)=0$. It is a differential module itself with
\begin{equation}\label{eq:dualdiff}
\partial_{M_\cL^*}:M_\cL^* \longrightarrow M_\cL^*, \qquad f\longmapsto \left[m \mapsto \partial(f(m)) - f\left(\partial_{M_\cL}(m)\right)\right]
\end{equation}
and a left $\Circ{U}$-comodule with
\begin{equation}\label{eq:coaction}
\rho_{M_\cL^*}:M_\cL^* \longrightarrow \Circ{U}\tensor{K}M_\cL^*,\qquad f \longmapsto \sum_{i=0}^{\deg(P_\cL)-1}\overline{f \tensor{T_{M_\cL}} y_i} \tensor{\KK} y_i^*.
\end{equation}

\begin{example}\label{exam:RR}
Here are some basic examples. Let $\R$ be the field of real numbers with zero derivation.
\begin{enumerate}[leftmargin=1cm]
\item Consider the differential equation $\partial\left(y\right)=yx$ for $x\in\R$. To it, one assigns the operator $\cL=\cN-t(x)$ on $\cH(\R)$ which corresponds to the polynomial $P_{\cL}(Y) = Y - x \in \R[Y;0]$. The space of solutions of $\cL$ is the one-dimensional vector space $\left(\R[Y]/P_{\cL}(Y)\R[Y]\right)^*$, which means that any solution is a scalar multiple of the one associated with the linear functional $\R[Y]/P_{\cL}(Y)\R[Y] \to \R, \overline{1}\mapsto 1$. By pre-composition with the canonical projection, we find $\R[Y] \to \R, Y^n \mapsto x^n$, corresponding to the sequence $\left(1,x,\ldots,x^n,\ldots\right)$, which in turn can be seen (via the algebra isomorphism $\R[Y]^*\cong \R[[T]]$) as the power series $\sum_{n\geq 0}\frac{x^n}{n!}T^n = \exp(xT)$. It is well-known that the solutions of $y'=yx$ are of the form $y=c\exp(xT)$ for $c\in \R$.
\item Consider the equation $\partial^2\left(y\right)+\omega^2 y=0$ with $\omega\neq 0$. The space of solutions of the associated operator $\cL = \cN^2 - t\left(\omega^2\right)$ is $\left(\R[Y]/\left(Y^2-\omega^2\right)\R[Y]\right)^*$, which is $2$-dimensional. Therefore, any solution is a linear combination of those corresponding to the morphisms $y_0^*$ and $y_1^*$. Now, by pre-composition with the canonical projection, $y_0^*$ corresponds to $\left(1,0,-\omega^2,0,\omega^4,0,-\omega^6,\ldots\right)$ while $y_1^*$ corresponds to $\left(0,1,0,-\omega^2,0,\omega^4,0,-\omega^6,\ldots\right)$. One would easily recognize the power series expansions of $\cos(\omega T)$ and $\frac{1}{\omega}\sin(\omega T)$, as expected from the theory of ordinary differential equations.
\end{enumerate}
\end{example}

Analogously, one may consider a differential polynomial $P(Y) \in U$ and set $M_{\Sscript{P}} \coloneqq {U}/{P(Y)U}$. If $\cL_P \coloneqq \sum_{i=0}^d t(c_i)\cN^i$ denotes the associated differential operator then $M_{\Sscript{P}}{}^*=\hom{\KK}{M_{\Sscript{P}}}{\KK} \cong \ker{\cL_P}$ can be seen as a sub-space of $\dlin{\KK}$. 

\begin{remark}\label{rem:nabla}
Let us introduce, for the sake of brevity, the component-wise derivation
\begin{equation}\label{eq:nablaexplicit}
\nabla:\cH(\KK) \longrightarrow \cH(\KK), \qquad \nabla(\beta) \coloneqq \left(\partial(\beta(n))\right)_{{n\geq 0}}
\end{equation}
and observe that the inverse of \eqref{eq:LaBrocante} is explicitly given by $M_{\cL}^{*} \to \ker{\cL}, f \mapsto \left(f(y_{n})\right)_{{n\geq 0}}$. Thanks to this, one may endow $\ker{\cL}$ with a structure of differential $\KK$-module with
\begin{equation}\label{eq:nabla}
\partial_{{\ker{\cL}}}(\alpha) = \nabla(\alpha) - \cN(\alpha)
\end{equation}
for every $\alpha \in \ker{\cL}$.
\end{remark}

The following proposition shows why $\Circ{U}$, or $\dlin{\KK}$, can be referred to as the universal space of solutions.

\begin{proposition}\label{prop:directlimit}
The family $\Big\{ M_P{}^* = \Hom{}{}{\KK}{}{M_P}{\KK} \Big\}_{\Sscript{P(Y)\, \in\,  U}}$, where $M_P\coloneqq U/P(Y)U$, forms a directed system of left $\Circ{U}$-comodules with morphisms $\pi_{P,Q}^*:M_P^* \to M_Q^*$ induced by the canonical right $U$-linear projections $\pi_{P,Q} : M_Q \to M_P$ for $Q(Y) \in P(Y)U$. Moreover, we have the following directed limit
\[
\injlimit{P\,\in\, U}{M_P{}^*}\, \cong\, \Circ{U}
\]
of left $\Circ{U}$-comodules with canonical injections being the comodule maps $M_{\Sscript{P}}{}^* \hookrightarrow \Circ{U}$, $f \mapsto \bara{f\tensor{T_{\Sscript{M_P}}} \bara{1} }$, where $\Circ{U}$ is a comodule over itself via the comultiplication $\Delta_\circ$ (see Remark \ref{rem:hopfalgd}). 
\end{proposition}
\begin{proof}
The morphisms $\pi_{P,Q}^*:M_P^* \to M_Q^*$ are colinear because of the following direct computation
\[
\sum_{i=0}^{\deg(Q)-1}\overline{\pi_{P,Q}^*(f) \tensor{T_{M_Q}} y_i} \tensor{\KK} y_i^* \stackrel{(*)}{=} \sum_{i=0}^{\deg(Q)-1}\overline{f \tensor{T_{M_P}} \pi_{P,Q}\left(y_i\right)} \tensor{\KK} y_i^* \stackrel{(**)}{=} \sum_{i=0}^{\deg(P)-1}\overline{f \tensor{T_{M_P}} y_i} \tensor{\KK} \pi_{P,Q}^*\left(y_i^*\right)
\]
where $(*)$ follows from the fact that $\pi_{P,Q}$ is a morphism of differential modules and hence it belongs to $T_{M_Q,M_P}$ and $(**)$ from the fact that the dual basis map $\KK \to M_P^* \tensor{\KK} M_P,\, 1_{\KK}\mapsto \sum_{i=0}^{\deg(P)-1} y_i \tensor{\KK} y_i^* $ satisfies
\[
\sum_{i=0}^{\deg(Q)-1} \pi_{P,Q}\left(y_i\right) \tensor{\KK} y_i^* = \sum_{i=0}^{\deg(P)-1} y_i \tensor{\KK} \pi_{P,Q}^*\left(y_i^*\right).
\]
Let us show that $\Circ{U}$ satisfies the universal property of the stated colimit. First of all, for every $P(Y)\in U$ consider the assignment $\psi_P:M_P^* \to \Circ{U}, f\mapsto \overline{f\tensor{T_{M_P}}\cl{1}}$. For all $x\in \KK$ and for all $f\in M_P^*$, we have that
\[
\psi_P\left(x\cdot f\right) = \psi_P\left(\lambda_x\circ f\right) = \overline{\lambda_x\circ f\tensor{T_{M_P}}\cl{1}} = x \cdot \big( \overline{f \tensor{T_{M_P}}\cl{1}}\big),
\]
whence $\psi_P$ is $\KK$-linear and
\[
\Delta_\circ\left(\psi_P(f)\right) = \sum_{i=0}^{\deg(P)-1} \overline{f \tensor{T_{M_P}} y_i} \tensor{\KK} \overline{y_i^* \tensor{T_{M_P}} \cl{1}} = \sum_{i=0}^{\deg(P)-1} \overline{f \tensor{T_{M_P}} y_i} \tensor{\KK} \psi_P\left(y_i^*\right) 
\]
whence it is left colinear. Moreover
\[
\psi_Q\left(\pi_{P,Q}^*\left(f\right)\right) = \psi_Q\left(f\circ \pi_{P,Q}\right) = \overline{f\circ \pi_{P,Q}\tensor{T_{M_Q}}\cl{1}} = \overline{f \tensor{T_{M_P}}\pi_{P,Q}\left(\cl{1}\right)} = \overline{f\tensor{T_{M_P}}\cl{1}} = \psi_P\left(f\right)
\]
for all $P(Y)\in U$, $Q(Y)\in P(Y)U$, $f\in M_P^*$. Thus, the $\psi_P$'s are compatible with the morphisms of the directed system, and by the isomorphism of Theorem \ref{thm:big} they are all injective.

Assume now that there exist a left $\Circ{U}$-comodule $V$ and $\Circ{U}$-colinear morphisms $\sigma_P : M_P^* \to V$ for all $P(Y)\in U$ such that $\sigma_Q \circ \pi_{P,Q}^* = \sigma_P$ for all $Q(Y)\in P(Y)U$. Now, for every finite-dimensional right $U$-module $M$ pick an element of the form $\varphi\tensor{T_M}m \in M^* \tensor{T_M} M$. Since $M$ is finite-dimensional over $\KK$, $m$ satisfies a relation of the form $m\triangleleft P_m(Y) = 0$ for a certain monic $P_m(Y)\in U$ (see Proposition \ref{prop:big1}). Thus the (unique) right $U$-linear morphism $\Phi_m:U\to M$ mapping $1$ to $m$ factors (uniquely) through a right $U$-linear morphism $\phi_m: M_{P_m} \to M, \cl{1}\mapsto m$. Define $\beta_M : M^* \tensor{T_M} M \to V, \varphi\tensor{T_M}m \mapsto \sigma_{P_m}\left(\varphi\circ \phi_m\right)$. This is well-defined because if $h\in T_M$, then
\[
0 = h\left(m\triangleleft P_m(Y)\right) = h\left(m\right)\triangleleft P_m(Y)
\]
and so $P_m(Y) \in P_{h(m)}(Y)U$. The induced morphism $\pi_{P_{h(m)},P_m}:U/P_m(Y)U \to U/P_{h(m)}(Y)U$ satisfies then $h\circ \phi_m = \phi_{h(m)} \circ \pi_{P_{h(m)},P_m}$ and hence
\[
\sigma_{P_m}\left(\left(\varphi \circ h\right)\circ \phi_m\right) = \sigma_{P_m}\left(\varphi \circ \phi_{h(m)} \circ \pi_{P_{h(m)},P_m}\right) = \left(\sigma_{P_m} \circ \pi_{P_{h(m)},P_m}^*\right)\left(\varphi \circ \phi_{h(m)} \right) = \sigma_{P_{h(m)}}\left(\varphi \circ \phi_{h(m)} \right).
\]
The family of morphisms $\beta_M$ for $M$ varying over all finite-dimensional right $U$-modules induces a unique morphism
\[
\beta' : \bigoplus_{M\in\cA_U} M^* \tensor{T_M} M \to V
\]
which factors through $\beta: \Circ{U} \to V$ by an argument similar to the one used to show that $\beta_M$ was well-defined.
Of course, $\beta\left(\psi_P(f)\right) = \beta\left(\overline{f\tensor{T_P}\cl{1}}\right) = \sigma_P(f)$ for all $P(Y)\in U$ and all $f\in M_P^*$. In addition, $\beta$ is the unique satisfying this property because we know that  
$$
\overline{\varphi \tensor{T_M} m} = \overline{\varphi\circ\phi_m \tensor{T_{M_{P_m}}}\cl{1}} = \psi_{P_m}\left(\varphi\circ\phi_m\right)
$$ for all $\overline{\varphi\tensor{T_M}m}\in\Circ{U}$. We are then left to check that $\beta$ is left $\Circ{U}$-colinear. To this aim, for every $\overline{\varphi \tensor{T_M} m} \in \Circ{U}$ pick a dual basis $\left\{e_i,e_i^*\mid i=1,\ldots,d\right\}$ of $M$ and compute
\begin{gather*}
\rho_V\left(\beta\left(\overline{\varphi \tensor{T_M} m}\right)\right) = \rho_V\left(\sigma_{P_m}\left(\varphi\circ\phi_m\right)\right) \stackrel{(*)}{=} \left(\Circ{U}\tensor{\KK}\sigma_{P_m}\right)\left(\rho_{M_{P_m}^*}\left(\varphi\circ\phi_m\right)\right) \\
\stackrel{\eqref{eq:coaction}}{=} \sum_{i=0}^{\deg(P_m)-1}\overline{\left(\varphi\circ\phi_m\right) \tensor{T_{M_{P_m}}} y_i} \tensor{\KK} \sigma_{P_m}\left(y_i^*\right) = \sum_{i=0}^{\deg(P_m)-1}\overline{\varphi \tensor{T_{M}} \phi_m\left(y_i\right)} \tensor{\KK} \sigma_{P_m}\left(y_i^*\right) \\
= \sum_{k=1}^{d}\overline{\varphi \tensor{T_{M}} e_k} \tensor{\KK} \sigma_{P_m}\left(\sum_{i=0}^{\deg(P_m)-1}e_k^*\left(\phi_m\left(y_i\right)\right)y_i^*\right) = \sum_{k=1}^{d}\overline{\varphi \tensor{T_{M}} e_k} \tensor{\KK} \sigma_{P_m}\left(e_k^*\circ\phi_m\right) \\
= \left(\Circ{U}\tensor{\KK}\beta\right)\left(\Delta_\circ\left(\overline{\varphi \tensor{T_M} m}\right)\right). \qedhere
\end{gather*}
\end{proof}

The subsequent result provides a constructive method for finding these solutions, given $d$ initial conditions $a_0,\ldots,a_{d-1}\in\KK$. It is noteworthy the role played in this by the universal coring structure on $\dlin{\KK}$. 

\begin{proposition}\label{prop:alpha}
A solution $\alpha\in\cH(\KK)$ of the equation $\cL=\sum_{i=0}^{d}t(c_i)\cN^i=0$, subject to the initial conditions $a_0,\ldots,a_{d-1}\in\KK$, is explicitly given by $\alpha(i)=a_i$ for $i=0,\ldots,d-1$ and
\[
\alpha(n) = \sum_{i=0}^{d-1}a_iy_i^*\left(y_n\right)
\]
for all $n\geq d$, where $\big\{y_i^*\big\}_{\Sscript{i=0,\,\ldots,\, d-1}}$ is the basis of the space of solutions $\left(U/P_\cL(Y)U\right)^*$ dual to the basis $\big\{y_i\coloneqq Y^i+P_\cL(Y)U\big\}_{\Sscript{i=0,\,\ldots,\, d-1}}$ of $M_{\Sscript{\cL}}\coloneqq U/P_\cL(Y)U$.
\end{proposition}

\begin{proof}
In view of Theorem \ref{thm:bullet}, we know that
\begin{gather*}
\alpha(h+k) = \xi(\alpha)\left(Y^{h+k}\right) = \sum_{(\alpha)} \xi\left(\alpha_{(1)}\right)\left(\xi\left(\alpha_{(2)}\right)\left(Y^h\right)Y^k\right) \\
= \sum_{j=0}^k\binom{k}{j}\xi\left(\alpha_{(1)}\right)\left(Y^j\right)\partial^{k-j}\left(\xi\left(\alpha_{(2)}\right)\left(Y^h\right)\right) = \sum_{j=0}^k\binom{k}{j}\alpha_{(1)}\left(j\right)\partial^{k-j}\left(\alpha_{(2)}\left(h\right)\right).
\end{gather*}
By writing this relation with $k=0$, $h=n$ and by resorting to the explicit description of $\Delta(\alpha)$ given in \eqref{eq:comult} we find out that
\[
\alpha(n) = \sum_{i=0}^{d-1}\xi\left(\cN^{i}(\alpha)\right)\left(1\right)\xi\left(\left(y_i^*\left(y_\bullet\right)\right)\right)\left(Y^n\right) = \sum_{i=0}^{d-1}a_iy_i^*\left(y_n\right)
\]
as claimed.
\end{proof}

\begin{remark}\label{rem:solutions}
An useful consequence of Proposition \ref{prop:alpha} is the following iterative method to construct formal solutions to homogeneous linear differential equations. Consider an equation of the form
\[
0 = \cL(y) = \partial^n\left(y\right) - (c_{n-1}\partial^{n-1}\left(y\right) + \cdots + c_1\partial\left(y\right) + c_0y)
\]
over a differential field $(\KK,\partial)$ as usual, and consider its extension
\[
0 = \cL(y) = \cN^{n}(y) - (t\left(c_{n-1}\right)\cN^{n-1}(y) + \cdots + t\left(c_1\right)\cN(y) + t\left(c_0\right)y)
\]
to $\left(\cH(\KK),\cN\right)$. Its space of solutions $M_\cL^*$ is an $n$-dimensional vector space over $\KK$ with basis $\left\{y_0^*,\ldots,y_{n-1}^*\right\}$ and we know, from Proposition \ref{prop:alpha}, that a full set of linearly independent formal solutions to $\cL(y)=0$ in $\cH(\KK)$ (in fact, in $\dlin{\KK}$) is provided by the sequences $o_i\coloneqq\left(y_i^*\left(y_\bullet\right)\right)$ for $i=0,\ldots,n-1$ and where we recall that $y_k = Y^k + P_\cL(Y)U \, \in U/ P_\cL(Y)U$, for all $k\geq 0$. For every $k\geq 0$ consider the column vector $\left(o_0\left(k\right), o_1\left(k\right),\ldots,o_{n-1}\left(k\right)\right)^T\in \KK^n$ and consider the matrix
\[
A\coloneqq\begin{pmatrix}
0 & 0 & \cdots & 0 & c_0 \\
1 & 0 & \cdots & 0 & c_1 \\
0 & 1 & \cdots & 0 & c_2 \\
\vdots & \vdots & \ddots & \vdots & \vdots \\
0 & 0 & \cdots & 1 & c_{n-1}
\end{pmatrix}.
\]
Then the solutions $o_i$ satisfy the recursive formula
\[
\begin{pmatrix}
o_0(k+1) \\
o_1(k+1) \\
\vdots \\
o_{n-1}(k+1)
\end{pmatrix} = 
\begin{pmatrix}
\partial\left(o_0(k)\right) \\
\partial\left(o_1(k)\right) \\
\vdots \\
\partial\left(o_{n-1}(k)\right)
\end{pmatrix} + 
A \cdot \begin{pmatrix}
o_0(k) \\
o_1(k) \\
\vdots \\
o_{n-1}(k)
\end{pmatrix}
\]
for all $k\geq n-1$, subject to the initial conditions
\[
\begin{pmatrix}
o_{0}(0) & o_0(1) & \cdots & o_0(n-1) \\
o_{1}(0) & o_1(1) & \cdots & o_1(n-1) \\
\vdots & \vdots & \ddots & \vdots \\
o_{n-1}(0) & o_{n-1}(1) & \cdots & o_{n-1}(n-1) \\
\end{pmatrix} = I_n,
\]
where $I_n$ is the identity $n\times n$ matrix.
\end{remark}

\begin{example}\label{ex:alphaBeta}
For a differential operator of degree two $\cL= \cN ^{2} - t(c_{\Sscript{1}}) \cN - t(c_{\Sscript{0}})$, the basis of the $\KK$-vector space of solutions inside $\dlin{\KK}$ is computed as follows. Consider  the attached differential polynomial $P_{\Sscript{\cL}}(Y) = Y^2-Yc_{\Sscript{1}} -c_{\Sscript{0}}$. Set  $y_{\Sscript{0}}\coloneqq\bara{1} = 1 +P_{\Sscript{\cL}}(Y)U$,  $y_{\Sscript{1}}\coloneqq y= Y +  P_{\Sscript{\cL}}(Y)U=\partial(\bara{1})$ (the basis as in Proposition \ref{prop:big2}) and for higher degree $y_n\coloneqq  Y^{n} +P_{\Sscript{\cL}}(Y)U=\partial^n(\bara{1})$, $n \geq 2$. Inside $\Circ{U}$ we have the following four elements 
$$
\bara{y_0^*\tensor{T_M}y_0},\quad  \bara{y_0^*\tensor{T_M}y_1},\quad \bara{y_1^*\tensor{T_M}y_0} \quad \text{  and } \;\; \bara{y_1^*\tensor{T_M}y_1}, 
$$
which, respectively, correspond to the following four sequences:
\begin{multline*}
\alpha^{0}=(1, 0, c_0, \partial\left(c_0\right)+c_0c_1, \cdots), \quad \cN(\alpha^0)= ( 0, c_0, \partial\left(c_0\right)+c_0c_1, \cdots), \\ \alpha^1= ( 0,1, c_1, c_1^2+c_0+\partial\left(c_1\right), \cdots),\quad \cN(\alpha^{1}) =(1, c_1, c_1^2+c_0+\partial\left(c_1\right), \cdots).
\end{multline*}
In matrix form, we have that 
$$
\begin{pmatrix}  \alpha^{0} \\ \\  \alpha^{1} \end{pmatrix} \,=\,  \begin{pmatrix} 1 & 0 & c_0 & \partial\left(c_0\right)+c_0c_1 & \cdots  & \alpha^{0}(n) & \cdots \\ & & & & & &   \\ 0  & 1 & c_1 & c_1^2+c_0+\partial\left(c_1\right) &  \cdots & \alpha^{1}(n) & \cdots   \end{pmatrix},
$$ 
with the following matrix recursive relations: 
$$
\begin{pmatrix}  \alpha^{0}(n) \\ \\  \alpha^{1}(n) \end{pmatrix} \,=\,  \begin{pmatrix}  \partial\left(\alpha^{0}(n-1)\right) \\ \\  \partial\left(\alpha^{1}(n-1)\right) \end{pmatrix} + \begin{pmatrix} 0 & c_0 \\ &    \\ 1 & c_1 \end{pmatrix} \begin{pmatrix}  \alpha^{0}(n-1) \\ \\  \alpha^{1}(n-1) \end{pmatrix},    \; \forall\, n\, \geq 2.
$$
We know from Proposition \ref{prop:alpha}, that $\{\alpha^0, \alpha^1\}$ leads to the full set of solutions of $\cL(y)=0$. We also know that $\beta^0=\cN(\alpha^0)$ and $\beta^{1}=\cN(\alpha^1)$ generate the space of solution of a differential equation of the same degree that we want to determine now.  For this reason, we will analyse the two cases: $c_0=0$ and $c_0\neq 0$. 

For the case $c_0=0$, we have that $\alpha^{0}=(1, 0, 0, \cdots)=s(1)$  and so $\cN(\alpha^{0})=\beta^{0}=0$, thus $\bara{y_0^*\tensor{T_M}y_1}=0$ in $\Circ{U}$. On the other hand,  we know that $\partial(y_1) -y_1c_1= 0$ and hence $\beta^{1}$ satisfies $\cN(\beta^{1}) -t(c_1) \beta^{1}=0$.  

For the case $c_0\neq0$, we follow explicitly the argument of the proof of Proposition \ref{prop:big1}. Since $y_0 = y_2/c_0 - y_1 c_1/c_0$, we have that 
$$
y_3\,=\, \partial(y_2) \,=\, y_2\big( c_1+\frac{\partial\left(c_0\right)}{c_0} \big) + y_1\big( \partial\left(c_1\right)+c_0-\frac{\partial\left(c_0\right)}{c_0} c_1\big).
$$
Therefore, 
$$
\partial^2(y_1) -\partial(y_1) b_1-y_1b_0\,=\,0, \quad \text{ with } \;   b_0= \partial\left(c_1\right)+c_0-\frac{\partial\left(c_0\right)}{c_0} c_1,\; b_1= c_1+\frac{\partial\left(c_0\right)}{c_0}. 
$$
As a consequence, $\bara{y_0^*\tensor{T_M}y_1}$ and $\bara{y_1^*\tensor{T_M}y_1}$ (and so $\beta^0$ and $\beta^1$ as well) satisfy the equation 
$$
\cN^2(y) -t(b_1) \cN(y) -t(b_0) y\,=\,0
$$
with $b_0, b_1$ as above. Observe that if $c_0, c_1 \in \K$ (i.e., they are constant elements),  then $b_0=c_0$ and $b_1=c_1$, and so $\beta^{0}, \beta^{1}$ satisfy the same recursive relation as $\alpha^{0}, \alpha^{1}$. 
\end{example}

\begin{example}\label{ex:order1}
Let $\left(\KK,\partial\right)\coloneqq (\C(z), \partial/\partial {z})$ be the field rational functions on $\C$ with the differential induced by the formal derivation with respect to $z$. On $\KK$, consider the general homogeneous linear differential equation of order one
\begin{equation}\label{eq:genord1}
\cL(y) = \partial\left(y\right)-ay=0
\end{equation}
for $a=u(z)/v(z)\in \C(z)$, $u(z),v(z)\in\C[z]$. Observe that the space of solutions of \eqref{eq:genord1} in a differential extension of $\C(z)$ is one-dimensional, because \eqref{eq:genord1} is of order one. In particular, if a non-zero solution belongs to some differential extension $(R,\partial_R)\supseteq (\KK,\partial)$, then $R$ contains a full set of solutions of \eqref{eq:genord1}. 

Being $\C$ algebraically closed, $v(z)=\prod_{i=1}^N\left(z-r_i\right)^{n_i}$ and hence, by the division algorithm in $\C[z]$, we have that $a$ can be rewritten as its partial fraction decomposition
\[
a=\sum_{i=1}^N\sum_{j=1}^{n_i}\frac{c_{i,j}}{\left(z-r_i\right)^j} + p(z)
\]
for certain $p(z)\in\C[z]$, $c_{i,j}\in\C$. It can be checked directly, by elementary arguments, that \eqref{eq:genord1} admits a non-zero solution in $\C(z)$ if and only if $c_{i,1}\in\ZZ$ for all $i=1,\ldots,N$, $c_{i,j}=0$ for all $i=1,\ldots,N$ and for all $j\geq 2$ and $p(z)=0$. It admits a non-zero solution which is algebraic over $\C(z)$ (\ie it satisfies a polynomial equation with coefficients in $\C(z)$) if and only if $c_{i,1}\in\QQ$ for all $i=1,\ldots,N$, $c_{i,j}=0$ for all $i=1,\ldots,N$ and for all $j = 2,\ldots,n_i$ and $p(z)=0$ (check \cite[Exercise 1.14(3)]{PutSinger}). 

Now, extend \eqref{eq:genord1} to $\left(\cH\left(\C(z)\right),\cN\right)$ via the differential morphism $t:\C(z) \to \cH\left(\C(z)\right)$. It becomes
\begin{equation}\label{eq:extgenord1}
\cL(y) = \cN(y)-t(a)y=0.
\end{equation}
Consider the associated polynomial $P_\cL(Y)=Y-a\in U=\KK[Y;\partial]$, the differential module
\[
M_\cL\coloneqq \frac{U}{P_\cL(Y)U}
\] 
with (right) $\KK$-basis $y_0\coloneqq 1+P_\cL(Y)U$ and differential
\[
\partial_\cL\left(u+P_\cL(Y)U\right) \coloneqq  \left(u+P_\cL(Y)U\right)\triangleleft Y = uY+P_\cL(Y)U, 
\]
and the associated left $\KK$-vector space $\left(U/P_\cL(Y)U\right)^*=\KK y_0^*$ of dimension one. If we set $y_n\coloneqq Y^n+P_\cL(Y)U$ for all $n\geq 0$, then there exists $o_n\in\C(z)$ (in fact, $o_n = y_0^*\left(y_n\right)$) such that $y_n = o_n + P_\cL(Y)U$ and hence
\begin{align*}
y_0o_{n+1} & = y_{n+1} = Y^{n+1}+P_\cL(Y)U = \left(Y^n+P_\cL(Y)U\right)\triangleleft Y = \left(o_n+P_\cL(Y)U\right)\triangleleft Y \\
& = \left(Yo_n+P_\cL(Y)U\right) + \left(\partial\left(o_n\right)+P_\cL(Y)U\right) = \left(Y+P_\cL(Y)U\right)\triangleleft o_n + \left(\partial\left(o_n\right)+P_\cL(Y)U\right) \\ & = \left(ao_n + \partial\left(o_n\right)\right)+P_\cL(Y)U \\
 & = y_0\left(ao_n + \partial\left(o_n\right)\right).
\end{align*}
By Proposition \ref{prop:alpha} we have that, for a given initial condition $\alpha(0)=p_0\in\C(z)$, the unique solution $\alpha$ of \eqref{eq:extgenord1} is given by $\alpha(n) = p_0o_n$ for all $n\geq 0 $ where $o_n$ satisfies the recursion
\[
\left\{
\begin{Array}[1.5]{l}
o_0 = 1, \\
o_{n+1} = ao_n + \partial\left(o_n\right),
\end{Array}
\right.
\]
that is to say,
\[
\alpha=s(p_0)\Big(1, a, a^2+\partial\left(a\right), a^3 + 3a\partial\left(a\right) + \partial^2\left(a\right), \ldots \Big).
\]
Set $o\coloneqq \left(o_\bullet\right)\coloneqq \left(y_0^*\left(y_\bullet\right)\right) = \left(1, a, a^2+\partial\left(a\right), \ldots \right)$, the solution corresponding to the initial condition $p_0=1$. In this more general framework, observe that \eqref{eq:genord1} admits a solution $f\in\C(z)$ if and only if $\partial\left(f\right)=af$, which is equivalent to say that $\partial^n\left(f\right) = o_nf$ for all $n\geq0$ (one implication is trivial, for the other one proceeds by induction on $n\geq1$). Thus, if and only if $t(f) = s(f)o$. 

Let us analyse the case $a=c/z$ for the sake of brevity, for some $c\in\C$. A direct computation by induction on $n\geq 1$ shows that
\begin{equation}\label{eq:power1}
o = \left(o_\bullet\right) = \left(1, \frac{c}{z}, \frac{c(c-1)}{z^2}, \frac{c(c-1)(c-2)}{z^3}, \ldots, \frac{c(c-1)\cdots(c-n+1)}{z^n},\ldots \right).
\end{equation}
Moreover, since $o$ satisfies $\cN(o)=t(a)o$, it follows that
\[
\cN(o^k) = ko^{k-1}\cN(o) = t(ka)o^{k},
\]
for all $k\in\ZZ\setminus\{0\}$, whence, by the same argument used to prove \eqref{eq:power1}, 
\begin{equation}\label{eq:kpower}
o^k = \left(1, \frac{kc}{z}, \frac{kc(kc-1)}{z^2}, \frac{kc(kc-1)(kc-2)}{z^3}, \ldots, \frac{kc(kc-1)\cdots(kc-n+1)}{z^n},\ldots \right).
\end{equation}
If $c\in\N$ then
\[
o = \left(1, \frac{c}{z}, \frac{c(c-1)}{z^2}, \ldots, \frac{c!}{z^c},0,\ldots \right) = \frac{1}{z^c}\left(z^c, cz^{c-1}, c(c-1)z^{c-2}, \ldots, c!,0,\ldots \right) = s\left(z^c\right)^{-1}t\left(z^c\right)
\]
and $\partial\left(y\right)-cy/z=0$ has general solution $f=\lambda z^c\in\C(z)$ for $\lambda\in\C$. If $c=-k$ for $k\in\N$ then
\begin{align*}
o & = \left(1, -\frac{k}{z}, \frac{k(k+1)}{z^2}, \ldots, (-1)^n\frac{k(k+1)\cdots(k+n-1)}{z^n},\ldots \right) \\
 & = z^k\left(\frac{1}{z^k}, -\frac{k}{z^{k+1}}, \frac{k(k+1)}{z^{k+2}}, \ldots, (-1)^n\frac{k(k+1)\cdots(k+n-1)}{z^{k+n}},\ldots \ldots \right) \\
 & = s\left(z^k\right)t\left(\frac{1}{z^k}\right)
\end{align*}
and $\partial\left(y\right)-cy/z=0$ has general solution $f=\lambda z^c\in\C(z)$ for $\lambda\in\C$. If $c=p/q\in\QQ$ with $p\in\ZZ$, $q\in\N\setminus\{0\}$, then 
\[
o^q \stackrel{\eqref{eq:kpower}}{=} \left(1, \frac{p}{z}, \frac{p(p-1)}{z^2}, \frac{p(p-1)(p-2)}{z^3}, \ldots, \frac{p(p-1)\cdots(p-n+1)}{z^n},\ldots \right),
\]
whence either $o^q = s\left(z^p\right)^{-1}t\left(z^p\right)$ (if $p>0$) or $o^q = s\left(z^p\right)t\left(z^p\right)^{-1}$ (if $p<0$). In both cases, $o$ is a solution of $\partial\left(y\right)-cy/z=0$ algebraic over $\C(z)$.
\end{example}

\begin{remark}
The following is a noteworthy relation arising from the computations performed in Example \ref{ex:order1}. For every $c\in\C$ and for every $k\in\N$, set formally
\[
\binom{c}{k}\coloneqq \frac{c(c-1)(c-2)\cdots(c-k+1)}{k!}.
\]
Then
\[
\sum_{k=0}^n\binom{c}{k}\binom{c}{n-k} = \binom{2c}{n}
\]
for all $n\geq 0$. Indeed, it is enough to compare term by term the explicit computation of $o^2$, by using the Hurwitz product, with formula \eqref{eq:kpower} for $k=2$. Even more general, for all $r,n\in\N$,
\[
\sum_{k_1+\cdots+k_r=n}\binom{c}{k_1}\binom{c}{k_2}\cdots\binom{c}{k_r} = \binom{rc}{n}.
\]
\end{remark}

\begin{example}\label{exm:C}
Let $(\KK,\partial)\coloneqq (\C(z), \partial/\partial {z})$ as in Example \ref{ex:order1}. Consider the following homogeneous differential equation:
\begin{equation}\label{Eq:C}
\partial^2\left(y\right) -\big(\frac{1}{z-1}\big)\partial\left(y\right)+ \big(\frac{1}{z-1}\big)^{2}y \,\,=\,\, 0.
\end{equation}
Denote by $c=\frac{1}{z-1} \in \KK$, so we have that $\partial(c)=-c^2$. The differential operator associated with \eqref{Eq:C} is $\cL= t(c^2) -t(c)\cN +\cN^2$ with corresponding polynomial $P_{\cL}(Y)=Y^2-Yc+c^2 \in U$. By Proposition \ref{prop:alpha}, a solution $\alpha \in \cH(\KK)$ of equation \eqref{Eq:C} subject to the initial conditions $a_{0}, a_{1}\in\KK$ has the form 
$$
\alpha(n)\,=\, a_{0} y_{0}^{*}(y_n) + a_{1} y_{1}^{*}(y_{n}).
$$
From the definition of the differential $\KK$-vector space $M_{\Sscript{\cL}}=U/P_{\cL}(Y)U$, we have the following recursive relation:
\[
\begin{pmatrix} y_{0}^{*}(y_{n+1})  \\ y_{1}^{*}(y_{n+1}) \end{pmatrix}\,=\,  \begin{pmatrix} \partial\big(y_{0}^{*}(y_{n}) \big)  \\ \partial\big(y_{1}^{*}(y_{n}) \big) \end{pmatrix} + \begin{pmatrix} 0 &  -c^2\\ 1 & c \end{pmatrix} \begin{pmatrix} y_{0}^{*}(y_{n})  \\ y_{1}^{*}(y_{n}) \end{pmatrix}.
\]
\begin{invisible}
In fact, set $y_n\coloneqq Y^n+P_\cL(Y)U$. Since $U/P_\cL(Y)U$ has dimension 2, generated by $y_0=1+P_\cL(Y)U$ and $y_1=Y+P_\cL(Y)U$, we have that for every $n\geq 2$ there exist $a_n,b_n\in\KK$ such that $y_n=y_0a_n+y_1b_n$. Therefore
\begin{gather*}
y_0a_{n+1}+y_1b_{n+1} = y_{n+1} = \left(Y^n+P_\cL(Y)U\right)\triangleleft Y = \left(y_0a_n+y_1b_n\right)\triangleleft Y = y_0'a_n+y_0a_n'+y_1'b_n+y_1b_n' \\
 = y_1a_n+y_0a_n'+\left(y_1c-y_0c^2\right)b_n+y_1b_n' = y_0\left(a_n'-c^2b_n\right) + y_1\left(b_n' + a_n + cb_n\right).
\end{gather*} 
After recalling that $a_n=y_0^*(y_n)$ and $b_n= y_1^*(y_n)$, the thesis follows.
\end{invisible}
In matrix form, the fundamental solutions (\ie those generating $\left(U/P_\cL(Y)U\right)^*$) can be expressed by:
\begin{equation}\label{eq:matrix}
\begin{pmatrix}           1 & 0 & -c^2 & c^3  &  -2 c^4   &  6 c^5 & \cdots  & (-1)^{n+1}(n-2)! c^n  & \cdots & \cdots     \\    &  & &   &     &  &   &   \\            0 & 1 & c & -c^2  &  2c^3   &  -6c^4 & \cdots  & (-1)^{n}(n-2)! c^{n-1}  & \cdots & \cdots        \end{pmatrix}
\end{equation}
The general solution $\alpha \in \cH(\KK)$ is then
$$
\begin{cases} \displaystyle \alpha(n) \,=\, (-1)^{n-1} \frac{(n-2)!\big(a_0-a_1(z-1)\big)}{(z-1)^{n}},\quad n \geq 2  \\  \\ \alpha(0)=a_{0},\, \alpha(1)=a_{1}. \end{cases}
$$
\end{example}

\begin{remark}
In Example \ref{exm:C} observe that, since $\partial\left(c\right)=-c^2$, $P_\cL(Y)=Y^2-Yc+c^2 = (Y-c)Y$. Therefore, any $\beta$ such that $\cN(\beta)-t(c)\beta=0$ satisfies $\cL(\beta)=0$ as well. Thanks to Example \ref{ex:order1}, we know that $\beta=s\left(b_0\right)\left(1,\frac{1}{z-1},0,\ldots\right)$ for some initial condition $b_0\in\KK$.  Now, by looking at \eqref{eq:matrix} the reader may easily convince himself that 
\[  
\beta = s(b_0) \Big(y_0^*\left(y_n\right) + c\, y_1^*\left(y_n\right)\Big)_{\Sscript{n \, \in \, \mathbb{N}}}, \quad b_0 \in \KK.
\]
as expected.
\end{remark}


\subsection{Comments on the Picard-Vessiot ring extension}   
In this section, we discuss the relation between the Picard-Vessiot differential ring extension of a given differential module and the Hopf algebroid of differentially recursive sequences.  For simplicity, we only treat the case of a differential module of rank two (that is, a differential $\KK$-vector space $(M,\partial_M)$ of dimension two).  We will often implicitly refer to notations and constructions from \cite{LaiachiGomez}.

Consider, as before, $(\KK, \partial)$ a differential field with $\Bbbk=\KK^{\partial}\subsetneq \KK$ its non-trivial sub-field of constant elements (assumed now to be algebraically closed of characteristic zero). 
Let us consider a linear homogeneous scalar differential equation
\begin{equation}\label{eq:Trescabrasyunamuerta}
\partial^2(y) - c_1\partial(y) - c_0y = 0,
\end{equation}
with $c_0, c_1 \in \KK$. After extending the latter equation to $(\cH(\KK),\cN)$ via the differential algebra homomorphism $t$, it corresponds to $\cL(\alpha)=0$ where the differential operator is given by $\cL=\cN^2 -t(c_1)\cN -t(c_0)$ and the associated differential polynomial by $P(Y)=Y^2-Yc_1-c_0$. Set $M \coloneqq U/P(Y)U$. It is a differential module $(M,\partial_{M})$ of dimension two over $\KK$ with dual basis $\{y_0, y_1, y_0^*, y_1^*\}$ and differential $\partial_M(y_n)=y_{n+1}$, where $y_n=Y^n+ P(Y)U$ for all $n \geq 0$. Therefore, $\partial_M(y_0)=y_1$ and $\partial_M(y_1)= c_0y_0+c_1y_1$. By considering the column expression (with the usual minus), the matrix of the differential $\partial_M$ computed as in \cite[page 7]{PutSinger} is then of the form   $\begin{pmatrix} 0 & -c_0 \\ -1 & -c_1  \end{pmatrix}$. 
Recall from Remark \ref{rem:nabla} that $\ker{\cL}\cong M^{*}$ as differential $\KK$-modules. The matrix of the differential module $(M^*, \partial_{M^*})$ is then the opposite of the transpose of the previous one. That is, we have $\partial_{M^*}(y_0^*)= -c_0y_1^*$ and $\partial_{M^*}(y_1^*)=-y_0^*-c_1y_1^*$. 
By Remark \ref{rem:solutions}, the dual basis for the solution space $M^*$ of $\cL(\alpha)=0$ over $\cH(\KK)$ satisfies the recursive relation:
\[
\begin{pmatrix} y_{0}^{*}(y_{n+1})  \\ \\ y_{1}^{*}(y_{n+1}) \end{pmatrix}\,=\,  \begin{pmatrix} \partial\big(y_{0}^{*}(y_{n}) \big)  \\ \\ \partial\big(y_{1}^{*}(y_{n}) \big) \end{pmatrix} + \begin{pmatrix} 0 &  c_0\\ & \\ 1 & c_1 \end{pmatrix} \begin{pmatrix} y_{0}^{*}(y_{n})  \\ \\ y_{1}^{*}(y_{n}) \end{pmatrix}, \quad \text{ for }  \, n \geq 0.
\]

In what follows, we will implicitly identify $M$ with $(M^*)^*$ in the rigid symmetric monoidal category of  differential modules over $\KK$.  
 Consider, as in Example \ref{ex:alphaBeta}, the following four elements 
$$
x_{00}\coloneqq \bara{y_0\tensor{T_{M^{*}}}y_0^{*}},\;  x_{01}\coloneqq \bara{y_0\tensor{T_{M^{*}}}y_1^{*}},\; x_{10}\coloneqq \bara{y_1\tensor{T_{M^{*}}}y_0^{*}} \; \text{ and }\, x_{11}\coloneqq \bara{y_1\tensor{T_{M^{*}}}y_1^{*}}
$$
in the Hopf $\KK$-algebroid $\Circ{U}$, that is, the differentially recursive sequences 
$$
\alpha^{0}=\big( y_0^*(y_n) \big)_{n \,\in \, \mathbb{N}},\; \alpha^{1}=\big( y_1^*(y_n) \big)_{n \,\in \, \mathbb{N}},\; \cN\big(\alpha^{0}\big),\, \text{ and } \, \cN\big(\alpha^{1}\big).
$$ 
Following \cite[Lemma 5.4.2]{LaiachiGomez}, the element $det(M^{*})=x_{00}x_{11}-x_{01}x_{10}$ is invertible in $\Circ{U}$ and it inverse is given by 
$$
det(M^{*})^{-1}\,=\, \bara{(y_0\wedge y_1)^* \tensor{T_{\bigwedge^2M}} (y_0\wedge y_1)} \, \in \, \Circ{U},
$$
where $\bigwedge^2M$ is the two-exterior power differential $\KK$-module of $(M,\partial_{M})$.

We denote by $\Uv$ the Hopf $\KK$-sub-algebroid of $\Circ{U}$ generated by the set of elements $\{x_{ij}, det(M^*)^{-1}\}_{0 \leq i,j \leq 1}$. It turns out that the Hopf algebroid $\Uv$ is in fact  the universal object constructed from the rigid monoidal full sub-category $\{\{M^*\}\}$ sub-quotient generated by $(M^*,\partial_M^*)$. Moreover, since we know that $\partial_\circ \circ s=0$ and $\partial_\circ \circ t = t \circ \partial$ (see Remark \ref{rem:hopfalgd}), the $(\KK\tensor{\Bbbk}\KK)$-algebra $\Uv$ is (via the target map) a differential extension of $(\KK, \partial)$ with differential the restriction of $\partial_\circ$. 

Let us denote by $\bd{\cP}$ the total isotropy Hopf $\KK$-algebra $\Uv/\langle s-t \rangle$, where $\langle s-t \rangle$ denote the Hopf ideal generated by the set $\{s(u)-t(u)\}_{u \, \in \, \Uv}$. In light of \cite[Proposition 5.5.2]{LaiachiGomez}, $\bd{\cP}$ is generated as a $\KK$-algebra by the elements:
$$
f_{ij}\coloneqq  x_{ij}+ \langle s-t \rangle, \quad 0 \leq i,j \leq 1,\quad \text{ and } \quad (f_{00}f_{11}-f_{01}f_{10})^{-1}.
$$
Moreover, it is a differential $\KK$-algebra whose differential $\delta: \bd{\cP} \to \bd{\cP}$ can be expressed by the rule
\begin{equation}\label{eq:delta}
\delta\left( \begin{pmatrix} f_{00} & f_{01} \\ f_{10} & f_{11} \end{pmatrix} \right) = \begin{pmatrix} f_{10} & f_{11} \\ c_0f_{00} + c_1f_{10} & c_0f_{01} + c_1f_{11} \end{pmatrix}= \begin{pmatrix} 0 & 1 \\ c_0 & c_1  \end{pmatrix}\begin{pmatrix} f_{00} & f_{01} \\ f_{10} & f_{11} \end{pmatrix} .
\end{equation}
Therefore the matrix $F\coloneqq (f_{ij})_{0 \leq i,j \leq 1}$ is a fundamental matrix (in the sense of \cite[Definition 1.9]{PutSinger}) for the linear differential matrix equation  attached to $(M^*,\partial_{M^*})$, with entries in $\bd{\cP}$. Furthermore, one can adapt the proof of \cite[Proposition 5.5.2]{LaiachiGomez} to show that $(\bd{\cP},\delta)$ is in fact a Picard-Vessiot ring of the differential $\KK$-vector space $(M^*,\partial_{M^*})$.  Notice that, as a differential $\KK$-algebra, $(\bd{\cP},\delta)$ is not an extension of the differential $\KK$-algebra $(\Uv,\partial_\circ)$, because the Hopf ideal $\langle s-t\rangle$ is not necessarily $\partial_\circ$-stable.  

Consider now the differential $\KK$-vector space $\left(\bd{\cP}\tensor{\KK}M^*, \partial_{\bd{\cP}\tensor{\KK}M^*}\right)$ with $\partial_{\bd{\cP}\tensor{\KK}M^*} = \delta\tensor{\KK} M^* + \bd{\cP} \tensor{\KK} \partial_{M^*}$ (see \cite[page 44]{PutSinger}). A direct check shows that the two elements 
$$
p_0\coloneqq  f_{00}\tensor{\KK}y_0^* + f_{10}\tensor{\KK}y_1^* \qquad \text{and} \qquad p_1\coloneqq f_{01}\tensor{\KK}y_0^* + f_{11}\tensor{\KK}y_1^*,
$$ 
in $\bd{\cP}\tensor{\KK}M^*$ are $\Bbbk$-linearly independent. Notice that $\{p_0,p_1\}$ generates the two dimensional $\Bbbk$-vector space $\ker{\partial_{\bd{\cP}\tensor{\KK}M^*}} \subseteq \bd{\cP}\tensor{\KK}V$, which is the solution space (in the sense of \cite[page 13]{PutSinger}) of the linear differential matrix equation  defined by $(M^*,\partial_{M^*})$. 

The assignment $\ker{\partial_{\bd{\cP}\tensor{\KK}M^*}} \to M^*, p_i \mapsto y_i^*$, clearly defines a monomorphism of $\K$-vector spaces. Moreover, when $c_0\neq 0$ one can show that the $\Bbbk$-vector space $\ker{\partial_{\bd{\cP}\tensor{\KK}M^*}}$ is isomorphism to the following  $\Bbbk$-subspace of $\bd{\cP}$:
\begin{equation}\label{Eq:KP}
\Big\{ p \in \bd{\cP}| \; \delta^2(p)= c_0p + (c_1+\frac{\partial\left(c_0\right)}{c_0})\delta(p)\Big\}.
\end{equation}

\begin{example}\label{Exam:Riccati}
Assume that $c_1=0$ and that $c_0\neq 0$ is a constant element of $\KK$. Take any non zero element $p$ in the sub-space described by equation \eqref{Eq:KP}, that is, $0\neq p \in \bd{\cP}$ such that $\delta^2(p)=c_0p$. Then the element $u=\delta(p)p^{-1}$ in the field of fractions of $\bd{\cP}$ is a solution of the equation $\partial\left(u\right) + u^2 = c_0$. The converse is also true when $\KK=\mathbb{C}(z)$ with the differential  $\partial/\partial z$. In this case the equation $\partial\left(u\right)+u^2=c_0$ is the so called Riccati equation. 
Assume now that $\KK=\C(z)$ and $c_1 \neq 0$. Denote by $\delta$ the differential on the field of fractions of $\bd{\cP}$ as well. More generally, if $p\in\bd{\cP}$ is a non-zero element of the Picard-Vessiot ring $\bd{\cP}$, then $u = \delta(p)/p$ satisfies
\[
\delta\left(\frac{\delta(p)}{p}\right) = \frac{\delta^2(p)p - \delta(p)^2}{p^2} = \frac{c_0p^2 + (c_1+\frac{\partial\left(c_0\right)}{c_0})p\delta(p) - \delta(p)^2}{p^2} = c_0 + \left(c_1+\frac{\partial\left(c_0\right)}{c_0}\right)\frac{\delta(p)}{p} - \left(\frac{\delta(p)}{p}\right)^2,
\]
that is to say, $u$ satisfies the (generalized) Riccati equation $\delta(u) = a(z) + b(z)u - u^2$ where $a(z) = c_0$ and $b(z) = c_1+\frac{\partial\left(c_0\right)}{c_0}$ (see \cite[\S I.1]{Reid}). Conversely, assume that $u$ is a solution of the Riccati equation and consider $y$ a solution of $\delta(y)=uy$. Then, we have
\[
\delta^2(y) = \delta(u)y+u\delta(y) = a(z)y + b(z)uy - u^2y+u\delta(y) + u^2y = c_0y + \left(c_1+\frac{\partial\left(c_0\right)}{c_0}\right)\delta(y).
\] 
\end{example}

\begin{remark}
Since $(M^*,\partial_{M^*})$ is the dual of $(M,\partial_M)$ in the category of differential modules, we have the bijective correspondences
\begin{equation}\label{eq:hom}
\hom{U}{\left(M,\partial_{M}\right)}{\left(\bd{\cP},\delta\right)} \cong \hom{U}{(\KK,\partial)}{\left(\bd{\cP}\tensor{\KK}M^*, \partial_{\bd{\cP}\tensor{\KK}{M^*}}\right)} \cong \ker{\partial_{\bd{\cP}\tensor{\KK}{M^*}}}
\end{equation}
(see, for example, \cite[page 45]{PutSinger}).
The distinguished differential morphisms in $\hom{U}{M}{\bd{\cP}}$ corresponding to the basis $\{p_0, p_1\}$ of $\ker{\partial_{\bd{\cP}\tensor{\KK}M^*}}$ under the isomorphism \eqref{eq:hom} are
\begin{equation}\label{Eq:M*P}
\begin{gathered}
\xymatrix@R=0pt@C=35pt{
q_0:M \; \ar@{^{(}->}[r] &   \Uv \ar@{->>}[r]  & \bd{\cP} \\  
y_0 \ar@{|->}[r]&  x_{00} \ar@{|->}[r]    & f_{00} \\ 
y_1 \ar@{|->}[r]  & x_{10} \ar@{|->}[r]   & f_{10}   
} 
\qquad  
\xymatrix@R=0pt@C=35pt{
q_1: M \; \ar@{^{(}->}[r] &   \Uv \ar@{->>}[r]  & \bd{\cP} \\  
y_0 \ar@{|->}[r]&  x_{01} \ar@{|->}[r]    & f_{01}   \\ 
y_1 \ar@{|->}[r]  & x_{11} \ar@{|->}[r]   & f_{11}    
}
\end{gathered}.
\end{equation}
Notice that the first one is induced by the canonical maps of Proposition \ref{prop:directlimit}.
\end{remark}

\begin{remark}\label{rem:Galoisgroup}
Under certain assumption (mainly on the generators $f_{ij}$ of the algebra $\cP$), one can connect the group of automorphisms of the differential vector space $(M, \partial_M)$ (which, in the sense of \cite{Keigher-diff},  is the dual space of solutions of equation \eqref{eq:Trescabrasyunamuerta}, see Remark \ref{rem:nabla}) with the differential group of the differential ring $(\cP, \delta)$, i.e., the group of $\KK$-algebras automorphisms of $\bd{\cP}$ that commute with the derivation $\delta$, which we  denote by ${\rm Aut}_{\Sscript{\text{alg-diff}}}((\cP,\delta))$. Precisely, if we assume that any invertible matrix $(\sigma_{ij})_{\Sscript{i,j}} \in GL_2(\KK)$ induces  a $\KK$-algebra automorphism of $\cP$ defined on the generator $f_{ij}$ by 
\begin{equation}\label{eq:sigma}
\sigma: \cP\longrightarrow \cP, \quad \Big(  (f_{ij})_{0\leq i,j\leq 1} \longmapsto  \begin{pmatrix} \sigma_{00} & \sigma_{01}\\ \sigma_{10} & \sigma_{11}  \end{pmatrix}  \begin{pmatrix} f_{00} & f_{01} \\ f_{10} & f_{11}  \end{pmatrix} \Big),
\end{equation}
then one shows that the group of automorphisms of the differential $\KK$-module $(M,\partial_{M})$ is in fact identified with a subgroup of the group ${\rm Aut}_{\Sscript{\text{alg-diff}}}((\cP,\delta))$. Specifically, let $\fk{g}$ be a $\KK$-linear automorphism of $M$ such that $\fk{g} \circ \partial_{M} = \partial_{M} \circ \fk{g}$. We use the above dual basis $\{y_0,y_1\}$ of $M$, and we set $\fk{g}(y_i)=g_{0i}y_0+ g_{1i}y_1$, $i=0,1$, for some $(g_{ij})_{\Sscript{0\leq i, j \leq 1}} \in GL_2(\KK)$. Then, one can easily check that the matrix $(g_{ij})_{\Sscript{0\leq i,j\leq 1}}$ satisfies
$$
\Big[\begin{pmatrix} 0 & c_0 \\ 1 & c_1 \end{pmatrix}  \, ,\, (g_{ij})_{\Sscript{i, j}}  \Big ]\,\,=\,\, (\partial g_{ij})_{\Sscript{i, j}},
$$ 
where the bracket stands for the Lie bracket (compare with \cite[Theorem 3.5]{Keigher-diff}).   Now under the above assumption, we have a well defined monomorphism of groups given by: 
$$
{\rm Aut}_{\text{diff}}((M^*, \partial_{M^*})) \longrightarrow {\rm Aut}_{\Sscript{\text{alg-diff}}}((\cP,\delta)), \quad \Big(  (g_{ij})_{\Sscript{i, j}} \longmapsto   (g_{ij})_{\Sscript{i, j}}^{T} \Big),
$$
where the matrix $(g_{ij})_ {\Sscript{0\leq i,j\leq 1}}^{T}$ is the transpose of $(g_{ij})_ {\Sscript{0\leq i,j\leq 1}}$, and  stands for an automorphism as in \eqref{eq:sigma}. 
\end{remark}


\vspace{0,5cm}

\noindent\textbf{Acknowledgements:}
Paolo Saracco expresses his heartfelt gratitude to the members of the Department of Algebra of the University of Granada for their warm hospitality and friendship during his stay in November-December 2019, when the greatest part of this work has been written.

\appendix

\section{Some technical details}\label{sec:appendix}

This appendix contains some results that we used along the paper but that we considered too technical for the main body. We report them here for the sake of completeness and of the unaccustomed reader.

\begin{lemma}\label{lemmaApp:vanish}
Let $\alpha\in\dlin{\KK}$ be a differentially recursive sequence of order $d$ (see Definition \ref{def:LRS}). If there exists an operator $\cL$ such that $\cL(\alpha)(n) = 0$ for all $n=0,\ldots,d-1$, then $\cL(\alpha)\equiv 0$. In particular, $\cL(\alpha) = 0$ if and only if $\cL(\alpha)(n) = 0$ for all $n=0,\ldots, d-1$.
\end{lemma}

\begin{proof}
If $\alpha\in\dlin{\KK}$ is a differentially recursive sequence of order $d$, then there exists an operator $\cL_\alpha\coloneqq \sum_{i=0}^d t(a_i)\cN^i$ such that $0 = \cL_\alpha(\alpha) = \alpha\triangleleft P_\alpha(Y)$, where $P_\alpha(Y) = \sum_{i=0}^d Y^ia_i$. Assume that $\cL' \coloneqq \sum_{i=0}^e t(c_i)\cN^i$ is another operator such that
\begin{equation}\label{eqApp:vanish'}
0 = \cL'(\alpha)(n) = \left(\alpha\triangleleft P'(Y)\right)(n) \qquad \text{for all} \quad n=0,\ldots,d-1.
\end{equation}
Consider $f_\alpha = \Phi^{-1}(\alpha) \in U^*$ and $\tilde{f}_\alpha \in (U/P_\alpha(Y)U)^*$. Equation \eqref{eqApp:vanish'}, together with \eqref{eq:PhiLin} and \eqref{eq:alphaf}, entails that
\begin{equation}\label{eqApp:zero}
\tilde{f}_\alpha\left(P'(Y)Y^n + P_\alpha(Y)U\right) = f_\alpha\left(P'(Y)Y^n\right) \stackrel{\eqref{eq:PhiLin}}{=} \Phi^{-1}\left(\alpha\triangleleft P'(Y)\right)(Y^n) \stackrel{\eqref{eq:alphaf}}{=} \cL'(\alpha)(n) \stackrel{\eqref{eqApp:vanish'}}{=} 0  
\end{equation}
for all $n=0,\ldots,d-1$. Let us prove by induction on $k\geq 0$ that $\cL'(\alpha)(d+k)=0$ as well. For the sake of simplicity, write $P_\alpha(Y) = Y^d - Q_\alpha(Y)$, with $\deg(Q_\alpha(Y))<d$. For $k=0$,
\begin{gather*}
\cL'(\alpha)(d) \stackrel{\eqref{eq:alphaf}}{=} f_{\cL'(\alpha)}\left(Y^d\right) \stackrel{\eqref{eq:PhiLin}}{=} f_\alpha\left(P'(Y)Y^d\right) = \tilde{f}_\alpha\left(P'(Y)Y^d + P_\alpha(Y)U\right) \\
 = \tilde{f}_\alpha\left(P'(Y)Q_\alpha(Y) + P_\alpha(Y)U\right) = \sum_{i=0}^{d-1} \tilde{f}_\alpha\left(P'(Y)Y^i + P_\alpha(Y)U\right)q_i \stackrel{\eqref{eqApp:zero}}{=} 0.
\end{gather*}
Now, assume that $\cL'(\alpha)(d+k)=0$ holds for all $k=0,\ldots,h-1$. Then
\begin{gather*}
\cL'(\alpha)(d+h) = \tilde{f}_\alpha\left(P'(Y)Y^{d+h} + P_\alpha(Y)U\right) = \tilde{f}_\alpha\left(P'(Y)Q_\alpha(Y)Y^{h} + P_\alpha(Y)U\right) \\
= \sum_{i=0}^{d-1} \tilde{f}_\alpha\left(P'(Y)Y^iq_iY^{h} + P_\alpha(Y)U\right)  \stackrel{\eqref{eq:useful}}{=} \sum_{i=0}^{d-1}\sum_{j=0}^{h}\binom{h}{j} \tilde{f}_\alpha\left(P'(Y)Y^{i+j} + P_\alpha(Y)U\right)\partial^{h-j}\left(q_i\right) \\
 = \sum_{i=0}^{d-1}\sum_{j=0}^{h}\binom{h}{j} \cL'(\alpha)(i+j)\partial^{h-j}\left(q_i\right) = 0
\end{gather*}
and so, by induction, $\cL'(\alpha)(n) = 0$ for every $n\in\N$.
\end{proof}

\begin{proposition}\label{propApp:supersysts}
Let $\alpha\in\cH(\KK)$ be a sequence and $d\geq 1$. The (non-linear) system of $d$ equations
\begin{equation}\label{eqApp:supersyst}
 0 = \cL\left(\alpha\right)(n) = \sum_{i=0}^{d}\sum_{k=0}^{n}\binom{n}{k}\partial^k\left(x_i\right)\alpha(n-k+i), \qquad 0\leq n\leq d-1,
\end{equation}
in the $d+1$ unknowns $x_0,\ldots,x_{d}$ is equivalent to the homogeneous (linear) system
\begin{equation}\label{eqApp:supersyst2}
\sum_{i=0}^{d}x_i\left(\sum_{k=0}^n\binom{n}{k}(-1)^k\partial^{k}\left(\alpha(n-k+i)\right)\right) = 0, \qquad 0\leq n\leq d-1.
\end{equation}
\end{proposition}

\begin{proof}
For every $n\in\N$, denote by $E_n$ the expression
\[
E_n \coloneqq \sum_{i=0}^{d}\sum_{k=0}^{n}\binom{n}{k}\partial^k\left(x_i\right)\alpha(n-k+i)
\]
and by $F_n$ the expression
\[
F_n \coloneqq \sum_{i=0}^{d}x_i\left(\sum_{k=0}^n\binom{n}{k}(-1)^k\partial^{k}\left(\alpha(n-k+i)\right)\right).
\]
Let us prove by induction on $n\geq 0$ that if $F_i = 0$ for all $i=0,\ldots,n$, then $E_{n+1} = F_{n+1}$. For $n = 0$, 
\[
\sum_{i=0}^dx_i\alpha(i) = E_0=F_0=0
\]
and 
\begin{gather*}
E_1 = \sum_{i=0}^dx_i\alpha(i+1)+\sum_{i=0}^d\partial\left(x_i\right)\alpha(i) = \sum_{i=0}^dx_i\alpha(i+1) - \sum_{i=0}^dx_i\partial\left(\alpha(i)\right) + \partial(E_0) = \sum_{i=0}^dx_i\left(\alpha(i+1) - \partial\left(\alpha(i)\right)\right) = F_1.
\end{gather*}
Assume that the claim holds for $k=0,\ldots,n-1$ and let us prove it for $k=n$. Since $F_s = 0$ for $s<n$, $\partial^{n-s}(F_s) = 0$, that is to say,
\begin{equation}\label{eqApp:tech1}
0 = \partial^{n-s}(F_s) = \sum_{i=0}^{d}\sum_{j=0}^{n-s}\sum_{k=0}^s\binom{s}{k}\binom{n-s}{j}(-1)^k\partial^j\left(x_i\right)\partial^{n-s-j+k}\left(\alpha(s-k+i)\right).
\end{equation}
Therefore, a very technical but otherwise straightforward computation shows that
\begin{gather*}
E_n \stackrel{\eqref{eqApp:tech1}}{=} E_n - \sum_{s=0}^{n-1}\binom{n}{s}\partial^{n-s}(F_s) \\
= \sum_{i=0}^{d}\sum_{k=0}^{n}\binom{n}{k}\partial^k\left(x_i\right)\alpha(n-k+i) - \sum_{s=0}^{n-1}\sum_{i=0}^{d}\sum_{j=0}^{n-s}\sum_{k=0}^s\binom{n}{s}\binom{s}{k}\binom{n-s}{j}(-1)^k\partial^j\left(x_i\right)\partial^{n-s-j+k}\left(\alpha(s-k+i)\right) \\
= \left[
\begin{gathered}
\sum_{i=0}^{d}\sum_{k=0}^{n}\binom{n}{k}\partial^k\left(x_i\right)\alpha(n-k+i) - \sum_{s=0}^{n-1}\sum_{i=0}^{d}\sum_{k=0}^s\binom{n}{s}\binom{s}{k}(-1)^kx_i\partial^{n-s+k}\left(\alpha(s-k+i)\right) + \\
 - \sum_{s=0}^{n-1}\sum_{i=0}^{d}\sum_{j=1}^{n-s}\sum_{k=0}^s\binom{n}{s}\binom{s}{k}\binom{n-s}{j}(-1)^k\partial^j\left(x_i\right)\partial^{n-s-j+k}\left(\alpha(s-k+i)\right) 
\end{gathered}
\right] \\
= \left[
\begin{gathered}
\sum_{i=0}^{d}\sum_{k=0}^{n}\binom{n}{k}\partial^k\left(x_i\right)\alpha(n-k+i) - \sum_{s=0}^{n-1}\sum_{i=0}^{d}\sum_{k=0}^s\binom{n}{s}\binom{s}{k}(-1)^kx_i\partial^{n-s+k}\left(\alpha(s-k+i)\right) + \\
 - \sum_{j=1}^{n}\sum_{s=0}^{n-j}\sum_{i=0}^{d}\sum_{k=0}^s\binom{n}{s}\binom{s}{k}\binom{n-s}{j}(-1)^k\partial^j\left(x_i\right)\partial^{n-s-j+k}\left(\alpha(s-k+i)\right) 
\end{gathered}
\right] \\
= \left[
\begin{gathered}
\sum_{i=0}^{d}x_i \left( \alpha(n+i) - \sum_{s=0}^{n-1}\sum_{k=0}^s\binom{n}{s}\binom{s}{k}(-1)^k\partial^{n-s+k}\left(\alpha(s-k+i)\right) \right) + \\
+ \sum_{i=0}^{d}\sum_{k=1}^{n}\partial^k\left(x_i\right)\left(\binom{n}{k}\alpha(n-k+i) - \sum_{s=0}^{n-k}\sum_{h=0}^s\binom{n}{s}\binom{s}{h}\binom{n-s}{k}(-1)^h\partial^{n-s-k+h}\left(\alpha(s-h+i)\right)\right) 
\end{gathered}
\right] \\
= \left[
\begin{gathered}
\sum_{i=0}^{d}x_i \left( \alpha(n+i) - \sum_{s=0}^{n-1}\sum_{k=0}^s\binom{n}{s}\binom{s}{k}(-1)^k\partial^{n-s+k}\left(\alpha(s-k+i)\right) \right) + \\
+ \sum_{i=0}^{d}\sum_{k=1}^{n}\binom{n}{k}\partial^k\left(x_i\right)\left(\alpha(n-k+i) - \sum_{s=0}^{n-k}\sum_{h=0}^s\binom{n-k}{s}\binom{s}{h}(-1)^h\partial^{n-s-k+h}\left(\alpha(s-h+i)\right)\right) 
\end{gathered}
\right].
\end{gather*}
Let us focus first on
\begin{gather*}
\alpha(n-k+i) - \sum_{s=0}^{n-k}\sum_{h=0}^s\binom{n-k}{s}\binom{s}{h}(-1)^h\partial^{n-s-k+h}\left(\alpha(s-h+i)\right) = \\
 = \alpha(q+i) - \sum_{s=0}^{q}\sum_{h=0}^s\binom{q}{s}\binom{s}{h}(-1)^h\partial^{q-s+h}\left(\alpha(s-h+i)\right) = \alpha(q+i) - \sum_{s=0}^{q}\sum_{k=0}^s\binom{q}{s}\binom{s}{s-k}(-1)^{s-k}\partial^{q-k}\left(\alpha(k+i)\right) \\
 = \alpha(q+i) - \sum_{k=0}^q\sum_{s=k}^{q}\binom{q}{s}\binom{s}{s-k}(-1)^{s-k}\partial^{q-k}\left(\alpha(k+i)\right) = \alpha(q+i) - \sum_{k=0}^q\left(\sum_{t=0}^{q-k}\binom{q}{t+k}\binom{t+k}{t}(-1)^{t}\right)\partial^{q-k}\left(\alpha(k+i)\right) \\
 = \alpha(q+i) - \sum_{k=0}^q\binom{q}{k}\left(\sum_{t=0}^{q-k}\binom{q-k}{t}(-1)^{t}\right)\partial^{q-k}\left(\alpha(k+i)\right) = \alpha(q+i) - \alpha(q+i) = 0
\end{gather*}
and then on
\begin{gather*}
\alpha(n+i) - \sum_{s=0}^{n-1}\sum_{k=0}^s\binom{n}{s}\binom{s}{k}(-1)^k\partial^{n-s+k}\left(\alpha(s-k+i)\right) = \alpha(n+i) - \sum_{s=0}^{n-1}\sum_{t=0}^s\binom{n}{s}\binom{s}{s-t}(-1)^{s-t}\partial^{n-t}\left(\alpha(t+i)\right) \\
 = \alpha(n+i) - \sum_{t=0}^{n-1}\sum_{s=t}^{n-1}\binom{n}{s}\binom{s}{s-t}(-1)^{s-t}\partial^{n-t}\left(\alpha(t+i)\right) = \alpha(n+i) - \sum_{t=0}^{n-1}\left(\sum_{h=0}^{n-1-t}\binom{n}{h+t}\binom{h+t}{h}(-1)^{h}\right)\partial^{n-t}\left(\alpha(t+i)\right) \\
 = \alpha(n+i) - \sum_{t=0}^{n-1}\binom{n}{t}\left(\sum_{h=0}^{n-1-t}\binom{n-t}{h}(-1)^{h}\right)\partial^{n-t}\left(\alpha(t+i)\right) \\
 = \alpha(n+i) - \sum_{t=0}^{n-1}\binom{n}{t}\left(\sum_{h=0}^{n-t}\binom{n-t}{h}(-1)^{h} - (-1)^{n-t}\right)\partial^{n-t}\left(\alpha(t+i)\right) \\
 = \alpha(n+i) + \sum_{t=0}^{n-1}\binom{n}{t}(-1)^{n-t}\partial^{n-t}\left(\alpha(t+i)\right) = \sum_{t=0}^{n}\binom{n}{t}(-1)^{n-t}\partial^{n-t}\left(\alpha(t+i)\right).
\end{gather*}
Thus,
\[
E_n = \sum_{i=0}^{d}x_i \left( \sum_{t=0}^{n}\binom{n}{t}(-1)^{n-t}\partial^{n-t}\left(\alpha(t+i)\right) \right) = F_n.
\]
This implies the following. If $\{x_i\}$ is a solution of $E_n = 0$ for $n=0,\ldots,d-1$ then $F_0 = E_0 =0$ and, by the inductive argument above, $F_n = E_n = 0$ for $n = 0,\ldots,d-1$. Conversely, If $\{x_i\}$ is a solution of $F_n = 0$ for $n=0,\ldots,d-1$ then, by the same argument, $E_n = F_n = 0$ for $n = 0,\ldots,d-1$.
\end{proof}


\end{document}